\documentclass{amsart}

\usepackage{amsmath}
\usepackage{amsfonts}
\usepackage{amssymb}
\usepackage{amsthm}
\usepackage{comment}
\usepackage{epsfig}
\usepackage{psfrag}
\usepackage{mathrsfs}
\usepackage{amscd}
\usepackage[all]{xy}
\usepackage{rotating}
\usepackage{lscape}
\usepackage{amsbsy}
\usepackage{verbatim}
\usepackage{moreverb}
\usepackage{color}
\usepackage{eucal}
\usepackage{stmaryrd}

\newtheorem{theorem}{Theorem}[section]
\newtheorem{prop}[theorem]{Proposition}
\newtheorem{lemma}[theorem]{Lemma}
\newtheorem{cor}[theorem]{Corollary}
\newtheorem{conj}[theorem]{Conjecture}

\theoremstyle{definition}
\newtheorem{definition}[theorem]{Definition}

\newtheorem{observation}[theorem]{Observation}

\newtheorem{terminology}[theorem]{Terminology}
\newtheorem{remark}[theorem]{Remark}
\newtheorem{example}[theorem]{Example}

\newtheorem{problem}[theorem]{Problem}
\newtheorem{notation}[theorem]{Notation}

\theoremstyle{remark}

\definecolor{orange}{rgb}{.95,0.5,0}
\definecolor{light-gray}{gray}{0.75}
\definecolor{brown}{cmyk}{0, 0.8, 1, 0.6}
\definecolor{plum}{rgb}{.5,0,1}

\DeclareMathOperator{\PShv}{\sf PShv}

\DeclareMathOperator{\colim}{{\sf colim}}

\DeclareMathOperator{\Fun}{{\sf Fun}}

\DeclareMathOperator{\Map}{{\sf Map}}

\DeclareMathOperator{\Cat}{{\sf Cat}}

\DeclareMathOperator{\fCat}{{\sf fCat}}
\DeclareMathOperator{\fGpd}{{\sf fGpd}}

\DeclareMathOperator{\sfN}{\sf fN}

\DeclareMathOperator{\Shv}{\sf Shv}

\DeclareMathOperator{\op}{\mathsf{op}}

\DeclareMathOperator{\cls}{\mathsf{cls}}

\DeclareMathOperator{\Spaces}{\cS\mathsf{paces}}

\DeclareMathOperator{\Bord}{\cB{\sf ord}}

\DeclareMathOperator{\Morita}{{\sf Morita}}
\DeclareMathOperator{\Corr}{{\sf Corr}}

\DeclareMathOperator{\oo}{\infty}

\DeclareMathOperator{\tr}{\triangleright}

\newcommand{\lag}{\langle}
\newcommand{\rag}{\rangle}

\newcommand{\un}{\underline}

\newcommand{\la}{\leftarrow}
\newcommand{\xra}{\xrightarrow}
\newcommand{\xla}{\xleftarrow}

\def\cB{\mathcal B}\def\cC{\mathcal C}\def\cD{\mathcal D}
\def\cE{\mathcal E}\def\cF{\mathcal F}\def\cG{\mathcal G}\def\cH{\mathcal H}
\def\cJ{\mathcal J}

\def\cS{\mathcal S}
\def\cV{\mathcal V}\def\cX{\mathcal X}
\def\cY{\mathcal Y}

\def\GG{\mathbb G}

\def\TT{\mathbb T}

\def\sC{\mathsf C}
\def\sE{\mathsf E}
\def\sL{\mathsf L}
\def\sN{\mathsf N}

\def\bTheta{\mathbf\Theta}

\def\fB{\frak B}

\def\bcT{\boldsymbol{\mathcal T}}

\DeclareMathOperator{\btheta}{\boldsymbol{\Theta}}

\DeclareMathOperator{\id}{{\sf id}}

\DeclareMathOperator{\Gpd}{\sf Gpd}

\DeclareMathOperator{\pr}{\sf pr}

\def\bDelta{\mathbf\Delta}

\newcommand{\nGpd}{n\Gpd}

\begin{document}

\title{Flagged higher categories}

\author{David Ayala \& John Francis}

\address{Department of Mathematics\\Montana State University\\Bozeman, MT 59717}
\email{david.ayala@montana.edu}
\address{Department of Mathematics\\Northwestern University\\Evanston, IL 60208}
\email{jnkf@northwestern.edu}
\thanks{DA was supported by the National Science Foundation under award 1507704. JF was supported by the National Science Foundation under award 1508040.}

\begin{abstract}
We introduce \emph{flagged $(\infty,n)$-categories} and prove that they are equivalent to Segal sheaves on Joyal's category $\bTheta_n$. As such, flagged $(\infty,n)$-categories provide a model-independent formulation of Segal sheaves.
This result generalizes the statement that $n$-groupoid objects in spaces are effective, as we explain and contextualize.  
Along the way, we establish a useful expression for the univalent-completion of such a Segal sheaf.
Finally, we conjecture a characterization of flagged $(\infty,n)$-categories as stacks on $(\infty,n)$-categories that satisfy descent with respect to colimit diagrams that do not generate invertible $i$-morphisms for any $i$.

\end{abstract}



\maketitle

\tableofcontents

\section*{Introduction}

Many examples of $(\infty,n)$-categories of especial interest, even for $n=1$, are univalent-completions of naturally presented Segal sheaves on the category $\bTheta_n$.
The following replacements occur as univalent-completions:
\begin{itemize}
\item a group by its moduli space of torsors, which loses conjugation information within the group;

\item a ring by its category of modules, which remembers only its Morita-type;

\item a category by its idempotent completion;

\item a suitably connective sequence $(X_0\to \dots \to X_n)$ of spaces by the space $X_n$ alone;

\item a smooth closed manifold by its smooth h-cobordism-type, which loses simple-homotopy-type.  
\end{itemize}
In light of this univalent-completion construction, it is in order to find a conceptual, model-independent, formulation of Segal sheaves on $\bTheta_n$.  
Our purpose for such a formulation is to accommodate native examples of such entities, and also to house such entities in a framework that can borrow results from established $(\infty,n)$-category theory (even for $n=1$).

In this paper we give a model-independent formulation of Segal sheaves on $\bTheta_n$, a corollary concerning higher groupoid objects, and conjecture another model-independent formulation; we state these three assertions informally here.
\begin{enumerate}
\item (Theorem~\ref{main.result.1}) 
A Segal sheaf on $\bTheta_n$ is equivalent to a flag $\cC_0 \to \cC_1\to \dots \to \cC_n$ in which each $\cC_i$ is an $(\infty,i)$-category and, for each $0\leq k\leq i\leq j\leq n$, the functor $\cC_i\to \cC_{j}$ is surjective on spaces of $k$-morphisms with specified source-target.

\item (Corollary~\ref{main.corollary})
An $n$-groupoid object in $\Spaces$ is precisely a flag $X_0 \to X_1 \to \dots \to X_n$ of spaces for which, for each $0\leq i\leq j\leq n$, the map $X_i \to X_j$ is $i$-connective.  

\item (Conjecture~\ref{main.result.2})
A Segal sheaf on $\bTheta_n$ is a stack on the $\infty$-category $\Cat_n$ of $(\infty,n)$-category that satisfies descent with respect to those colimit diagrams that do not generate invertible $i$-morphisms for any $0\leq i \leq n$.  

\end{enumerate}

\subsection*{Conventions}
We make use of Lurie's work~\cite{HTT}, as well as Joyal's work~\cite{joyal1}, for the foundations of $\infty$-category theory -- there, quasi-category theory.  
This includes a comprehensive theory of colimits, limits, (space-valued) presheaves,  the Yoneda embedding as a colimit completion, the unstraightening construction, and Bousfield localizations among presentable $\infty$-categories.
We assume the reader has operational, though not necessarily technical, acquaintance with these features of $\infty$-category theory.
We also call on some more specific features of the $\infty$-category $\Spaces$, which are consequences of the fact that it is an $\infty$-topos in the sense of~\S6 of~\cite{HTT}.  We assume the reader has a working acquaintance with Joyal's category $\bTheta_n$, as it is presented in Berger's work~\cite{berger} as well as Rezk's work~\cite{rezk-n}.  We assume the reader has a working acquaintance with $(\infty,n)$-categories as developed by Rezk in~\cite{rezk-n}.

We make use of the following notation.
\begin{notation}\label{d5}
\begin{itemize}
\item[]

\item
We may denote the colimit of a presheaf $\cF\colon \cC^{\op} \to \Spaces$ as $|\cF| := \colim(\cF)$.  

\item
Let $\cC$ be an $\infty$-category.
By right Kan extension, each presheaf $(\cC^{\op} \xra{\cF} \Spaces)\in \PShv(\cC)$ extends along the Yoneda embedding as a functor
\[
\cF\colon \PShv(\cC)^{\op}
\xra{~\Map(-,\cF)~}
\Spaces
~,\qquad
\cE  \mapsto \cF(\cE):=\Map(\cE,\cF)~.
\]

\end{itemize}

\end{notation}

\subsection{Setup and main results}

We give a definition of Joyal's category $\btheta_n$~(\cite{joyal.theta}), which follows Definition~3.9 in~\cite{berger}.
\begin{definition}\label{d7}
The category $\bTheta_n$, and its subcategory $\bTheta_n^{\sf cls}\hookrightarrow \bTheta_n$ of \emph{closed} morphisms, are defined by induction on $n$ as follows.  
\begin{itemize}
\item
For $n<0$, $\bTheta_{n} := \emptyset$.
Assume $n\geq 0$.
An object in the category $\bTheta_n$ is a pair of objects $[p]\in \bDelta$ and 
$(S_1,\dots,S_p)\in (\bTheta_{n-1})^{\times p}$; such an object is typically denoted $[p](S_1,\dots,S_p)$.
A morphism in $\bTheta_n$ from $[p](S_1,\dots,S_p)$ to $[q](T_1,\dots,T_q)$ is a morphism $ [p]\xra{\sigma}[q]$ in $\bDelta$ together with, for each $0<i\leq p$ and $\sigma(i-1)<j\leq \sigma(i)$, a morphism $S_i \xra{\tau_{ij}} T_j$ in $\bTheta_n$.  
Composition of morphisms in $\bTheta_n$ is given by composing morphisms in $\bDelta$ and composing collections of morphisms in $\bTheta_{n-1}$.  

\item
For each $n\geq 0$, the subcategory $\bTheta_n^{\sf cls}\hookrightarrow \bTheta_n$ contains all objects, and for $n>0$, only those morphisms $[p](S_1,\dots,S_p) \xra{\bigl(\sigma , (\tau_{ij})\bigr)}[q](T_1,\dots,T_q)$ for which $\sigma$ is a consecutive inclusion and each $\tau_{ij}$ is a morphism in $\bTheta_{n-1}^{\sf cls}$.  

\end{itemize}
For $0\leq k\leq n$, the \emph{$k$-cell} is $c_k:=[1](c_{k-1})$ for $k>0$ and $c_0 = [0]$.  

\end{definition}

\begin{remark}\label{r7}
Consider the $(2,1)$-category $\Cat_n^{\sf strict}$ of \emph{strict $n$-categories}, by which it is meant the $(2,1)$-category of ordinary categories enriched over the Cartesian monoidal $(2,1)$-category of strict $(n-1)$-categories.  
There is a fully faithful functor $\bTheta_n \hookrightarrow \Cat_n^{\sf strict}$, as established in~\cite{berger}.
The \emph{nerve} functor is the restricted Yoneda functor along this fully faithful functor:
\[
\Cat_n^{\sf strict}
\longrightarrow
\PShv(\bTheta_n)
~,\qquad
\sC  \mapsto   \bigl( T\mapsto \Cat_n^{\sf strict}(T , \sC) \bigr)  ~.
\]

\end{remark}

\begin{notation}\label{d3}
Let $0\leq i\leq n$.
The strict $i$-category $c_{i-1} \wr \sE(c_1)$ corepresents an invertible $i$-morphism in a strict $n$-groupoid.
The presheaf on $\bTheta_n$ which is its nerve is given the same notation. This is a special case of Definition~\ref{d1}.
\end{notation}

Notice the morphism between presheaves on $\bTheta_n$,
\begin{equation}\label{e33}
c_{i-1} \wr \sE(c_1) 
\longrightarrow
c_{i-1}     ~,
\end{equation}
corepresenting identity $i$-morphisms as invertible $i$-morphisms within strict $n$-categories.

\begin{definition}\label{def.segal.cov}
\begin{itemize}
\item[~]

\item
A \emph{Segal cover in $\bTheta_n$} is a colimit diagram $\cJ^{\tr}\to \bTheta_n^{\sf cls}$.
The $\infty$-category of \emph{Segal sheaves (on $\bTheta_n$)} is the full $\infty$-subcategory 
\[
\Shv(\bTheta_n)
~\subset~
\PShv(\bTheta_n)
\]
consisting of those presheaves that carry (the opposites of) Segal covers to limit diagrams.

\item
For $0<i\leq n$, the \emph{$i$-univalence diagram in $\bTheta_n$} is the functor $\bigl( {\bTheta_n}_{/c_{i-1}\wr \sE(c_1)} \bigr)^{\tr} \to \bTheta_n$ which is adjoint to~(\ref{e33}).  
The $\infty$-category of \emph{univalent Segal sheaves (on $\bTheta_n$)} is the full $\infty$-subcategory 
\[
\Shv^{\sf unv}(\bTheta_n)
~\subset~
\Shv(\bTheta_n)
\]
consisting of those Segal sheaves that carry (the opposites of) univalence diagrams to limit diagrams.

\end{itemize}
\end{definition}

\begin{remark}
We use the notation $\Shv(\bTheta_n)$ for Segal sheaves on $\bTheta_n$ to suggestively regard the Segal condition on a presheaf as a descent condition with respect to a notion of a cover. 
We warn the reader, however, that Segal covers do \emph{not} form a Grothendieck topology on $\bTheta_n$.
Likewise, the $\infty$-category $\Shv(\bTheta_n)$ it is not an $\infty$-topos; therefore, there is no Grothendieck site for which $\Shv(\bTheta_n)$ is its $\infty$-category of sheaves.
Nevertheless, the Segal covers of Definition~\ref{def.segal.cov} \emph{do} define a Grothendieck topology on the subcategory $\bTheta_n^{\sf cls}\subset \bTheta_n$.
Said another way, the pullback $\infty$-category in the diagram
\[
\xymatrix{
\Shv(\bTheta_n^{\sf cls})  \ar[rr]  \ar[d]
&&
\Shv(\bTheta_n)  \ar[d]
\\
\PShv(\bTheta_n^{\sf cls})  \ar[rr]
&&
\PShv(\bTheta_n)
}
\]
is an $\infty$-topos.  
In fact, this $\infty$-topos is free on its infinitesimal basics:
\[
\PShv(\GG_n^{\sf cls})~\simeq~\Shv(\bTheta_n^{\sf cls})    ~,
\]
where $\GG_n^{\sf cls} \subset \bTheta_n^{\sf cls}$ is the full subcategory consisting of the cells.  

\end{remark}

We recall the following culminating definition of~\cite{rezk-n}.

\begin{definition}[\cite{rezk-n}]\label{def.rezk}
The $\infty$-category $\Cat_n$ of $(\infty,n)$-categories is initial among presentable $\infty$-categories under $\bTheta_n$,
\[
\bTheta_n \longrightarrow \Cat_n~,
\]
that carry Segal covers to colimit diagrams and carry univalence diagrams to colimit diagrams.  

\end{definition}

\begin{observation}\label{t23}
From their defining universal properties, there is a canonical identification between $\infty$-categories under $\bTheta_n$: 
\[
\Cat_n~\simeq~ \Shv^{\sf unv}(\bTheta_n)   ~.
\]
\end{observation}

\begin{definition}\label{def.connective}
Let $0\leq i\leq n$.
A functor $\cC \to \cD$ between $(\infty,n)$-categories is \emph{$i$-connective} if, for each $0\leq k\leq i$, each solid diagram among $(\infty,n)$-categories
\[
\xymatrix{
\partial c_k  \ar[rr]  \ar[d]  
&&
\cC  \ar[d]
\\
c_k  \ar[rr]  \ar@{-->}[urr]
&&
\cD
}
\]
can be filled.

\end{definition}

\begin{example}
Let $\cX \to \cY$ be a map between $\infty$-groupoids, and let $n\geq 0$ be an integer.
Regarded as a functor between $(\infty,n)$-categories, it is $i$-connective if and only if it is $i$-connective as a map between spaces.

\end{example}

\begin{remark}\label{r2}
Let $\cC\to \cD$ be a functor between $(\infty,n)$-categories.
One might say this functor is \emph{$k$-surjective} if it is surjective on spaces of $k$-morphisms with specified source-target.
More precisely, if, for each functor $\partial c_k \to \cC$, the resulting map between spaces $\Map^{\partial c_k/}(c_k , \cC) \to \Map^{\partial c_k/}(c_k, \cD)$ is surjective on path components.  
Through this terminology, the functor $\cC\to \cD$ being $i$-connective is equivalent to it being $k$-surjective for each $0\leq k\leq i$.  

\end{remark}

\begin{definition}\label{def.flagged.n.cat}
A \emph{flagged $(\infty,n)$-category} is a sequence of morphisms among $(\infty,n)$-categories
\[
\cC_0\longrightarrow \cC_1 \longrightarrow \cdots \longrightarrow \cC_n
\]
satisfying the following conditions:
\begin{itemize}
\item
for each $0\leq i \leq n$, the $(\infty,n)$-category $\cC_i$ is actually an $(\infty,i)$-category;

\item
for each $0\leq i \leq j \leq n$, the functor $\cC_i \to \cC_{j}$ is $i$-connective.

\end{itemize}
The $\infty$-category of \emph{flagged $(\infty,n)$-categories} is the full $\infty$-subcategory
\[
\fCat_n~\subset~\Fun\bigl([n],\Cat_n\bigr)
\]
consisting of the flagged $(\infty,n)$-categories.
\end{definition}

\begin{example}
In general, a flagged $(\infty,1)$-category is an $(\infty,1)$-category $\cC$, together with a \emph{surjective} functor $\cG\to \cC$ from an $\infty$-groupoid.  Here \emph{surjective} means \emph{essentially surjective}, or equivalently it means $\pi_0$-surjective on spaces of objects.  
\end{example}

\begin{example}
Let $A$ be an associative algebra in the Cartesian symmetric monoidal $\infty$-category $\Spaces$.
Its deloop $\ast \to \fB A$ is an $\infty$-category equipped with a functor from the terminal $\infty$-groupoid which is surjective on maximal $\infty$-subgroupoids.  
\end{example}

\begin{example}
More generally, let $A$ an $\cE_n$-algebra $A$ in the Cartesian symmetric monoidal $\infty$-category $\Spaces$.
Its $n$-fold deloop $\ast \to \dots \to \ast \to \fB^n A$ is an $(\infty,n)$-category equipped wiht a functor from the terminal $(\infty,n-1)$-category which is $(n-1)$-connective.
\end{example}

\begin{example}
Consider the ordinary category $\Morita$ whose objects are associative rings, whose morphisms from $A$ to $B$ are $(B,A)$-bimodules, and whose composition rule is given as follows: for $P$ a $(B,A)$-bimodule, and for $Q$ a $(C,B)$-bimodule, the composition $Q\circ P$ is the $(C,A)$-bimodule
$
P\underset{B}\otimes Q
$.
Equivalences in $\Morita$ are Morita equivalences between rings.  
In particular, for each commutative ring $R$, the objects ${\sf Mat}_{2\times 2}(R)$ and $R$ are equivalent in $\Morita$.  
Consider the flagged $(\infty,1)$-category
\[
{\sf Rings}^\sim \longrightarrow \Morita~.
\]
The underlying $\infty$-groupoid of this flagged $(\infty,1)$-category is, by design, 
that of isomorphisms between associative rings.  
\end{example}

\begin{example}
More generally, consider the $(2,1)$-category $\Corr$ whose objects are ordinary categories and whose morphisms from $\cC$ to $\cD$ are $(\cD,\cC)$-bimodules, and whose composition rule is given by coend.
Two categories are equivalent in $\Corr$ if their idempotent completions are equivalent as categories.  
In particular, the ordinary category corepresenting an idempotent and that corepresenting a retraction are equivalent in $\Corr$ yet they are not equivalent as categories -- there is a unique fully faithful epimorphism between them, and it is not surjective.  
Consider the flagged $(\infty,1)$-category 
\[
{\sf Cat}^\sim \longrightarrow \Corr~.
\]
The underlying $\infty$-groupoid of this flagged $(\infty,1)$-category is, by design, that of equivalences between ordinary categories.
\end{example}

\begin{example}
We make use of the terminology introduced in Remark~\ref{r2}.
In general, a flagged $(\infty,2)$-category is an $(\infty,2)$-category $\cC_2$, together with a functor $\cC_1\to \cC_2$ from an $(\infty,1)$-category that is 0-surjective and 1-surjective, together with a 0-surjective functor $\cC_0 \to \cC_1$ from an $\infty$-groupoid.  
\end{example}

\begin{example}
Let $\cG$ be an $(\infty,0)$-category, which is simply a space.
Denote by $\cG_{\leq -1}$ its $(-1)$-truncation: this is a space which is initial if $\cG$ is empty, and is final otherwise.  
Then $\cG \to \cG_{\leq -1}$ is a flagged $(\infty,1)$-category.
This construction is present in the definition of an enriched $\infty$-category, as developed in~\cite{david.rune}.  
Namely, for $\cV$ a monoidal $\infty$-category, the canonical functor to its deloop $\ast \to \fB \cV$ is a flagged $(\infty,1)$-category (internal to $\Cat$).  
A $\cV$-enriched $\infty$-category, with underlying $\infty$-groupoid $\cC_0$, is a lax functor between flagged $(\infty,1)$-categories ${\sf hom}_\cC\colon (\cC_0 \to (\cC_0)_{\leq -1}) \xra{\sf lax} (\ast \to \fB \cV)$ satisfying a certain univalence condition.  
\end{example}

\begin{example}\label{r3}
A flagged $(\infty,n)$-category in which each constituent $(\infty,i)$-category is, in fact, an $\infty$-groupoid is precisely a flag of spaces $(X_0\to \dots \to X_n)$ in which each map $X_i\to X_{j}$ is $i$-connective.  
(For an interesting example, consider a knot $K\subset S^3$, and take $X_0=X_1=K$ and $X_2=X_3=S^3$.
For another interesting example, consider a point $\ast \in X$ in an $n$-connective space and take $X_0=\dots=X_{n-1}=\ast$ and $X_n = X$.) 
The underlying $(\infty,n)$-category of this flagged $(\infty,n)$-category is the $\infty$-groupoid $X_n$, which is blind to the maps $X_i\to X_n$.
In the example coming from a point $\ast \in X$ in an $n$-connective space, this flagged $(\infty,n)$-category is the $\cE_n$-algebra $\Omega^n X$, whereas the $(\infty,n)$-category associated to this flagged $(\infty,n)$-category is its $n$-fold deloop $X$, as an unpointed space.  
In the case $n=1$, the space of automorphisms of $X$ is the space of outer automorphisms of the $\cE_1$-algebra $\Omega X$.
\end{example}

\begin{example}
Two closed $(n-1)$-manifolds are equivalent as $(n-1)$-morphisms in the $(\infty,n)$-category $\Bord_n$ if and only if they are h-cobordant (see~\S2.2 of~\cite{cobordism} for a discussion of this).
Consider the fantastic example of a flagged $(\infty,n)$-category
\[
\Bord_0 \longrightarrow \Bord_1\longrightarrow \cdots \longrightarrow \Bord_n
\]
given by the standard functors.  
The underlying $(\infty,i)$-category in this flagged $(\infty,n)$-category, by design, codifies diffeomorphisms between compact $i$-manifolds (with corner structure).  
\end{example}

\begin{observation}\label{maximals}
Consider the standard sequence of fully faithful right adjoint functors
\begin{equation}\label{e5}
\bTheta_0 
~\hookrightarrow~ 
\bTheta_1 
~\hookrightarrow~
\dots 
~\hookrightarrow~ 
\bTheta_{n-1}
~\hookrightarrow~
\bTheta_n~.
\end{equation}
Each of the functors in this sequence, as well as their left adjoints, preserves Segal covers and univalence diagrams.
Therefore, left Kan extension along each functor in the above sequence determines a sequence of fully faithful left adjoint functors
\[
\Cat_0
~\hookrightarrow~
\Cat_1
~\hookrightarrow~
\dots 
~\hookrightarrow~
\Cat_{n-1}
~\hookrightarrow~
\Cat_n~;
\]
the right adjoint to each of these functors is given by restriction along the corresponding functor in~(\ref{e5}).
\end{observation}

\begin{terminology}\label{def.mxml.sub}
For each $0\leq i \leq j \leq n$, the value of the right adjoint to $\Cat_i \hookrightarrow \Cat_j$ on an $(\infty,j)$-category $\cC$ is its \emph{maximal $(\infty,i)$-subcategory} $\cC_{\leq i}\subset \cC$.  
\end{terminology}

\begin{observation}\label{maximal.underlying}
Evaluation at the target defines a left adjoint
\[
\fCat_n \longrightarrow \Cat_n
\]
in a localization between $\infty$-categories.
The right adjoint carries an $(\infty,n)$-category $\cC$ to the flagged $(\infty,n)$-category 
\[
\cC_{\leq 0} \longrightarrow \cC_{\leq 1} \longrightarrow \dots \longrightarrow \cC_{\leq n-1}\longrightarrow \cC_{\leq n}=\cC  ~.
\]
\end{observation}

The fully faithful functors
\begin{equation}\label{e13}
\bTheta_n
~\hookrightarrow~ 
\Cat_n
~\hookrightarrow~ 
\fCat_n
\end{equation}
determine the restricted Yoneda functors
\begin{equation}\label{0}
\fCat_n \longrightarrow \PShv(\Cat_n)
\qquad
\text{ and }
\qquad
\sfN\colon \fCat_n \longrightarrow \PShv(\bTheta_n)~.
\end{equation}
In light of the factorization~(\ref{e13}), the functor $\sfN$ extends the standard nerve functor:
\[
\sN\colon \Cat_n \hookrightarrow \fCat_n \xra{~\sfN~} \PShv(\bTheta_n)~.
\]

\begin{remark}
Let $\un{\cC} = (\cC_0 \to \cC_1 \to \dots \to \cC_n)$ be a flagged $(\infty,n)$-category.
The value of the presheaf $\sfN(\un{\cC})$ on $T\in \bTheta_n$ is the space of fillers in the commutative diagram among $(\infty,n)$-categories:
\[
\xymatrix{
T_{\leq 0}  \ar[r]  \ar@{-->}[d]  
&
T_{\leq 1}  \ar[r]  \ar@{-->}[d]  
&
\cdots  \ar[r]  
&
T_{\leq n-1}  \ar[r] \ar@{-->}[d]  
&
T_{\leq n}  \ar@{-->}[d]   
\\
\cC_0  \ar[r]
&
\cC_1 \ar[r]
&
\cdots   \ar[r]
&
\cC_{n-1}  \ar[r]
&
\cC_n  .
}
\]
In the case that each canonical functor to the maximal $(\infty,i)$-subcategory $\cC_i \to (\cC_n)_{\leq i}$ is an equivalence, such a diagram is just the data of its rightmost vertical arrow.  
\end{remark}

Here is our main result, which we prove in~\S\ref{sec.main.proof}.
\begin{theorem}\label{main.result.1}
The restricted Yoneda functor
\[
\sfN\colon \fCat_n \xra{~(\ref{0})~} \PShv(\bTheta_n)
\]
is fully faithful, with image consisting of those presheaves that carry (the opposites of) Segal covers to limit diagrams:
\[
\sfN\colon \fCat_n\xra{~\simeq~} \Shv(\bTheta_n)~.
\]
\end{theorem}

\begin{remark}
Theorem~\ref{main.result.1} offers a model-independent description Segal sheaves on $\bTheta_n$, as we explain.
By Definition~\ref{def.rezk}, the $\infty$-category $\Cat_n$ of $(\infty,n)$-categories is defined via a universal property that references $\bTheta_n$.  
The work of Barwick--Schommer-Pries~(\cite{clark-chris}) articulates a sense in which this dependence on $\bTheta_n$ can be relieved, or rather replaced by an assortment of other basic categories $\bcT_n$.
In this way, we regard the $\infty$-category $\Cat_n$ of $(\infty,n)$-categories as \emph{model-independent} -- it can be described as a full $\infty$-subcategory of presheaves on an assortment of basic categories $\bcT_n$, not just $\bcT_n=\bTheta_n$.  
Supported by this, the Definition~\ref{def.flagged.n.cat} of the $\infty$-category of flagged $(\infty,n)$-categories, then, is a model-independent notion.
Theorem~\ref{main.result.1} therefore gives a model-independent description of Segal sheaves on $\bTheta_n$.
\end{remark}

\subsection{A corollary}
We draw a corollary of Theorem~\ref{main.result.1}, as it specializes to the case of groupoids.
To state this corollary, we give two auxiliary definitions.

Recall from Notation~\ref{d3} the strict $n$-category $c_{i-1} \wr \sE(c_1)$.
Consider the functor between strict $n$-categories,
\begin{equation}\label{e25}
c_i 
\longrightarrow 
c_{i-1} \wr \sE(c_1)   ~,
\end{equation}
corepresenting invertible $i$-morphisms as examples of $i$-morphsims within strict $n$-categories.
\begin{definition}\label{d6}
The $\infty$-category of \emph{$n$-groupoid objects (in $\Spaces$)} is the full $\infty$-subcategory
\[
\nGpd[\cS]
~\subset~
\PShv(\bTheta_n)
\]
consisting of those presheaves $\cG\colon \bTheta_n^{\op} \to \Spaces$ that satisfy the following conditions.
\begin{enumerate}
\item
$\cG$ carries (the opposites of) Segal diagrams to limit diagrams.

\item
For each $0<i\leq n$, each solid diagram among presheaves on $\bTheta_n$ admits a filler:
\[
\xymatrix{
c_i \ar[rr]  \ar[d]_-{(\ref{e25})}
&&
\cG  
\\
c_{i-1} \wr \sE(c_1)  \ar@{-->}[urr]_-{\exists}
&&
.
}
\]
\end{enumerate}
\end{definition}

\begin{remark}
Let $0<i\leq n$.
The morphism $c_i \xra{(\ref{e25})} c_{i-1}\wr \sE(c_1)$ in $\Shv(\bTheta_1)$ is an epimorphism.
So, for each Segal sheaf $\cF$ on $\bTheta_n$, the map between spaces induced by $\cF\bigl( c_{i-1}\wr \sE(c_1)\bigr) \xra{ \cF(\ref{e25})} \cF(c_i)$ is a monomorphism.
Therefore, condition~(2) in Definition~\ref{d6} is equivalent to the condition that, for each $0<i\leq n$, this monomorphism $\cG\bigl( c_{i-1} \wr \sE(c_1) \bigr) \xra{\simeq} \cG(c_i)$ is surjective on components, and is therefore an equivalence.  
In the case that $\cG$ is the nerve of a strict $n$-category, this condition~(2) is exactly the condition that $\cG$ is, in fact, a strict $n$-groupoid.
This justifies the terminology of Definition~\ref{d6}.  
\end{remark}

The next definition isolates the examples discussed in Example~\ref{r3}.  
\begin{definition}\label{d4}
The $\infty$-category of \emph{$n$-flagged $\infty$-groupoids} is the full $\infty$-subcategory 
\[
\fGpd_n~\subset~\Fun([n],\Spaces)
\]
of sequences $(X_0 \to \dots \to X_n)$ for which the map $X_i \to X_j$ is $i$-connective for each $0\leq i\leq j \leq n$.  

\end{definition}

\begin{observation}\label{t19}
There are evident fully faithful functors
\[
\nGpd[\cS]
~\hookrightarrow~
\Shv(\bTheta_n)
\qquad\text{ and } \qquad
\fGpd_n
~\hookrightarrow~
\fCat_n~.
\]
\end{observation}

We isolate the following consequence of Theorem~\ref{main.result.1}, which is of independent interest.  
Its proof occupies~\S\ref{sec.groupoid}.
\begin{cor}\label{main.corollary}
The equivalence of Theorem~\ref{main.result.1} restricts as an equivalence between $\infty$-categories:
\[
\xymatrix{
\fGpd_n  \ar@{-->}[rr]^-{\sfN}_-{\simeq}    \ar[d]_-{\rm Obs~\ref{t19}}
&&
\nGpd[\cS]  \ar[d]^-{\rm Obs~\ref{t19}}
\\
\fCat_n  \ar[rr]^-{\sfN}_-{\underset{\rm Thm~\ref{main.result.1}}{\simeq}}
&&
\Shv(\bTheta_n)  .
}
\]

\end{cor}

\begin{remark}\label{r4}
Let $\un{X} = (X_0 \to X_1 \to \dots \to X_n)$ be an $n$-flagged $\infty$-groupoid.  
Through Corollary~\ref{main.corollary}, this is equivalent data to the $n$-groupoid object $\sfN(\un{X})$.
This $n$-groupoid object $\sfN(\un{X})$ is the Segal sheaf on $\bTheta_n$ with the following values on cells and their boundaries.
\begin{itemize}

\item
$\sfN(\un{X})(c_0) \simeq X_0$.  Also, $\sfN(\un{X})(\partial c_1) \simeq  X_0\times X_0$.  

\item
$\sfN(\un{X})(c_1) \simeq X_0 \underset{X_1}\times X_0$.
Also, $\sfN(\un{X})(\partial c_2) \simeq \bigl(X_0 \underset{X_1}\times X_0\bigr) \underset{X_0\times X_0} \times \bigl( X_0 \underset{X_1}\times X_0\bigr)$.

\item
\begin{eqnarray}
\nonumber
\sfN(\un{X})(c_2) 
&
\simeq 
&
X_2 \underset{ \bigl(X_2 \underset{X_2}\times X_2\bigr) \underset{X_2\times X_2} \times \bigl( X_2 \underset{X_2}\times X_2\bigr)}\times  \bigl(X_0 \underset{X_1}\times X_0\bigr) \underset{X_0\times X_0} \times \bigl( X_0 \underset{X_1}\times X_0\bigr)
\\
\nonumber
&
\simeq
&
X_2 \underset{X^{S^1}}\times  \sfN(\un{X})(\partial c_2)   ~ .
\end{eqnarray}

\item
In general, there are pullback squares among spaces:
\[
\xymatrix{
\sfN(\un{X})(c_i)  \ar[r]  \ar[d]  
&
X_i  \ar[d]^-{\sf diag}
&&
\sfN(\un{X})(\partial c_j)  \ar[r]  \ar[d]  
&
\sfN(\un{X})(c_{j-1})     \ar[d]
\\
\sfN(\un{X})(\partial c_i)  \ar[r]  
&
X_i^{S^{i-1}}    
&
\text{ and }
&
\sfN(\un{X})(c_{j-1})   \ar[r]    
&
\sfN(\un{X})(\partial c_{j-1})       .
}
\]

\end{itemize}
Informally, an object in $\sfN(\un{X})$ is a point in $X_0$.
A 1-morphism in $\sfN(\un{X})$ is a path in $X_1$ equipped with lifts of its endpoints to $X_0$.
A 2-morphism in $\sfN(\un{X})$ is a 2-disk in $X_2$, equipped with compatible lifts of its hemispheres to $X_1$ and lifts of its poles to $X_0$. 
Continuing, an $i$-morphism is an $i$-disk in $X_i$ equipped with compatible lifts of its hemispherical $j$-strata to $X_j$ for $j<i$.  
Informally, composition is given by concatenating disks.    
\end{remark}

\begin{remark}
Let $\un{X}$ be an $n$-flagged $\infty$-groupoid.  
The connectivity assumptions on $\un{X}$ ensure that the space $X_n$ can be recovered as the colimit of the $n$-groupoid object $\sfN(\un{X})$.
Without these connectivity assumptions, this colimit $|\sfN(\un{X})|$ would report a suitable connective cover of $X_n$. 
\end{remark}

\begin{remark}\label{r5}
Let $\un{X} = (X_0 \to X_1)$ be a 1-flagged $\infty$-groupoid. 
The $1$-groupoid object $\sfN(\un{X})$, which is in particular a simplicial space, is the \emph{Cech nerve} of the map $X_0 \to X_1$, in the sense of~\S6.1.2 of~\cite{HTT}.  
In this sense, Corollary~\ref{main.corollary} generalizes the fact that every $1$-groupoid object $\cG_\bullet$ in $\Spaces$ is the Cech nerve of the canonical map $\cG_0 \to |\cG_\bullet|$ to its colimit.
\end{remark}

\begin{remark}
Let $\un{X} = (X_0 \to \dots \to X_n)$ be an $n$-flagged $\infty$-groupoid.
After Remark~\ref{r5}, the $n$-groupoid object $\sfN(\un{X})$, which is in particular a presheaf on $\bTheta_n$, can be interpreted as the \emph{$n$-Cech nerve} of the given flag $\un{X}$.
In this way, Corollary~\ref{main.corollary} states that every $n$-groupoid object $\bTheta_n^{\op}\xra{\cG_\bullet}\Spaces$ in $\Spaces$ is the $n$-Cech nerve of the canonical flag of maps $\cG_0 \to |\cG_{\bullet\leq 1}| \to | \cG_{\bullet \leq 2}| \to \dots \to | \cG_{\bullet \leq n} |$.  
\end{remark}

\begin{remark}\label{r6}
Remark~\ref{r5} can be interpreted as an instance of unstable Koszul duality over the $\cE_1$-operad, which we expand on now.
Let $\un{X} = (X_0 \to X_1)$ be a 1-flagged $\infty$-groupoid.  
Fix a field $\Bbbk$ of characteristic 0; consider the presentable $\infty$-category ${\sf Stack}(\Bbbk)$ of (commutative) $\Bbbk$-stacks.
The functor $\ast \xra{\lag {\sf Spec}(\Bbbk) \rag} {\sf Stack}(\Bbbk)$, which selects the terminal (commutative) $\Bbbk$-stack, uniquely extends as a colimit preserving functor $\Spaces \to {\sf Stack}(\Bbbk)$.  
In this way, each space and each diagram among spaces, determines a (commutative) $\Bbbk$-stack and a diagram of (commutative) $\Bbbk$-stacks, respectively.  
In particular, the given map $X_0 \to X_1$ between spaces determines a map between (commutative) $\Bbbk$-stacks.  
Let us suppose that $X_0 \simeq \ast$ is terminal.  
$\cE_1$-deformations of this map between (commutative) $\Bbbk$-stacks is organized as a functor from local Artin $\cE_1$-$\Bbbk$-algebras.
Koszul duality asserts that this functor is represented by an augmented $\cE_1$-algebra $\TT^{\cE_1}_\ast X_1$.  
For formal reasons, this representing augmented $\cE_1$-$\Bbbk$-algebra $\TT^{\cE_1}_\ast X_1$ is the universal enveloping $\cE_1$-algebra of the Lie algebra $\TT_\ast X_1$, which is the tangent space of the (commutative) $\Bbbk$-stack $X_1$ at its point.  
For other formal reasons, this representing augmented $\cE_1$-$\Bbbk$-algebra $\TT^{\cE_1}_\ast X_1$ is the group ring $\Bbbk[\Omega X] := \sC_\ast(\Omega X_1 ; \Bbbk)$ on the group $\Omega X_1$ which $\sfN(\un{X})$ codifies, as it is equipped with its standard augmentation.  
Conversely, using the assumption that the map between spaces $\ast \to X_1$ is $0$-connective gives that the canonical map from the Maurer--Cartan $\cE_1$-$\Bbbk$-stack,
\[
{\sf MC}_{\Bbbk[\Omega X_1]} =: {\sf Spec}(\Bbbk)_{/\Bbbk[\Omega X_1]} 
\xra{~\simeq~}
X_1     ~,
\]
is an equivalence (as $\cE_1$-$\Bbbk$-stacks).
Though less developed, we anticipate a similar interpretation of Corollary~\ref{main.corollary} for the general $n=1$ case (in which $X_0$ is general). 
Specifically,
\begin{itemize}
\item
the 1-groupoid $\sfN(\un{X})$, which we regard as an unstable version of an $\cE_1$-algebroid over $X_0$, represents $\cE_1$-deformations of the map $X_0 \to X_1$ between (commutative) $\Bbbk$-stacks;

\item
the connectivity of the map $X_0 \to X_1$ between spaces ensures that $X_1$, as an $\cE_1$-$\Bbbk$-stack, is the Maurer--Cartan $\cE_1$-$\Bbbk$-stack of this $\cE_1$-algebroid over $X_0$.

\end{itemize}
\end{remark}

\begin{remark}
We follow-up on Remark~\ref{r6}.
Though even less developed, we speculate a further interpretation of Corollary~\ref{main.corollary} for the case of general $n$.
Specifically, for $\un{X} = (X_0 \to \dots \to X_n)$ an $n$-flagged $\infty$-groupoid, 
\begin{itemize}
\item
the n-groupoid $\sfN(\un{X})$, which we regard as an unstable version of an $\cE_1$-algebroid over an $\cE_1$-algebroid over ... over an $\cE_1$-algebroid over $X_0$, represents compatible $\cE_1$-deformations of each map $X_i \to X_{i+1}$ in the given flag $\un{X}$;

\item
the connectivity of each map $X_i \to X_{j}$ ensures that $X_n$, as an $\cE_n$-$\Bbbk$-stack, is the Maurer--Cartan $\cE_n$-$\Bbbk$-stack of this iterated $\cE_1$-algebroid over $X_0$.

\end{itemize}
\end{remark}

\subsection{A conjecture}
We state a conjecture, and some related problems, that are prompted by this work.
To state our conjecture we single out a class of diagrams in $\Cat_n$.  

\begin{definition}\label{def.gaunt.diag}
A \emph{gaunt colimit diagram in $\Cat_n$} is a functor $\cJ^{\tr} \to \Cat_n$ for which, for each $0\leq i \leq n$, the composite functor $\cJ^{\tr} \to \Cat_n \xra{(-)_{\leq i}} \Cat_i$ is a colimit diagram.  

\end{definition}

\begin{example}
For each $0<i<p$, the diagram
\[
\xymatrix{
\{i\}    \ar[r]  \ar[d]
&
\{i<\dots<p\}  \ar[d]
\\
\{0<\dots<i\}  \ar[r]
&
\{0<\dots<p\}
}
\]
is a gaunt colimit diagram in $\Cat_1$.
More generally, for each Segal cover $\cJ^{\tr} \to \bTheta_n$, the composite functor
\[
\cJ^{\tr} \longrightarrow \bTheta_n \longrightarrow \Cat_n
\]
is a gaunt colimit diagram.  

\end{example}

\begin{example}
While the diagram in $\Cat_1$
\[
\xymatrix{
\{0<2\}\amalg \{1<3\}  \ar[r] \ar[d]
&
\{0<1<2<3\} \ar[d]
\\
\ast\thinspace\amalg \thinspace\ast \ar[r]
&
\ast
}
\]
is a colimit diagram, it is \emph{not} a gaunt colimit diagram.  
Note, however, that the map from the colimit of maximal $\infty$-subgroupoids to the maximal $\infty$-subgroupoid of $\ast$,
$\infty$-groupoids
\[
\Bigl(\{0<1<2<3\}^\sim \underset{\{0<2\}^\sim\amalg \{1<3\}^\sim}\coprod \ast^\sim \amalg \ast^\sim   \Bigr)
~\simeq~ \{-\}\amalg \{+\} \longrightarrow \ast~,
\]
is $0$-connective, which is to say that it is surjective on path components.  

\end{example}

\begin{example}
Consider the ordinary 1-category $\sE(c_1)$ corepresenting an isomorphism.  
Consider the functor ${\bDelta}_{/\sE(c_1)}:= \bDelta\underset{\Cat_1}\times (\Cat_1){/\sE(c_1)} \to \bDelta\hookrightarrow \Cat_1$ from the slice category.  
Consider its terminal extension
\[
({\bDelta}_{/\sE(c_1)})^{\tr} \longrightarrow \Cat_1~.
\]
While this functor is a colimit diagram, it is \emph{not} a gaunt colimit diagram.  
More generally, for each $0<i\leq n$, consider the strict $n$-groupoidification $\sE(c_i)$ of the $i$-cell.
While the composite functor
\[
({\bTheta_n}_{/\sE(c_i)})^{\tr} \longrightarrow \bTheta_n ~\hookrightarrow~\Cat_n
\]
is a colimit diagram, it is \emph{not} a gaunt colimit diagram.

\end{example}

\begin{remark}
Heuristically, a colimit diagram $\cJ^{\tr}\to \Cat_n$ is \emph{gaunt} if it does not generate invertible $i$-morphisms for any $i$.

\end{remark}

We make the following.
\begin{conj}\label{main.result.2}
The restricted Yoneda functor
\[
\fCat_n \xra{~(\ref{0})~} \PShv(\Cat_n)
\]
is fully faithful, with image consisting of those presheaves that carry (the opposites of) gaunt colimit diagrams to limit diagrams.  

\end{conj}

\begin{remark}\label{reps}
We reflect on the statement of Conjecture~\ref{main.result.2}.  
By definition, the $\infty$-category $\Cat_n$ is presentable.  
Therefore, the image of the Yoneda functor $\Cat_n \to \PShv(\Cat_n)$, which is fully faithful, consists of those presheaves that carry (the opposites of) \emph{all} colimit diagrams to limit diagrams.  

\end{remark}

\begin{problem}
Identify checkable criteria for when a colimit diagram $\cJ^{\tr} \to \Cat_n$ is in fact a gaunt colimit diagram.

\end{problem}

\begin{problem}
Let $0\leq i \leq n$.
For $\partial c_i \to \cC$ a functor to a finite gaunt $n$-category, give an explicit description (for instance as a presheaf on $\bTheta_n$, or even as a presheaf on finite gaunt $n$-categories) of the pushout $(\infty,n)$-category
\[
\cC\underset{\partial c_i} \amalg c_i~.
\]

\end{problem}

\subsection*{Acknowledgements} 
We are grateful to Jacob Lurie, for his contributions to foundational topos theory, and to Charles Rezk, for his careful and operationally practical exposition of $(\infty,n)$-categories.

\section{Comparing Segal sheaves and flagged higher categories}
We establish adjunctions connecting the $\infty$-categories $\Shv(\bTheta_n)$ and $\fCat_n$.  
In doing so, we introduce some interpolating $\infty$-categories.

\subsection{An adjunction}

Recall the sequence of functors~(\ref{e5}); for each $0\leq i\leq j\leq n$, denote that functor as
\begin{equation}\label{e6}
\iota := \iota_{ij} \colon \bTheta_i \hookrightarrow \bTheta_j
\end{equation}
where the subscripts are omitted if the indices are understood from context.  
The sequence~(\ref{e5}) of fully faithful right adjoint functors is selected by a functor:
\[
\bTheta_\bullet\colon [n]
\xra{\lag (\ref{e5}) \rag}
\Cat 
~,\qquad
i\mapsto \bTheta_i~.
\]
Postcomposing this functor with the functor $\PShv\colon \Cat \to {\sf Pr}^\sL$, to presentable $\infty$-categories and left adjoint functors among them, results in a functor $[n]\xra{\bTheta_\bullet} \Cat\xra{\PShv} {\sf Pr}^\sL$.  
The unstraightening of this functor is an $\infty$-category over $[n]$,
\begin{equation}\label{e7}
\PShv(\bTheta_\bullet) 
\longrightarrow
[n]~,
\end{equation}
which is both a coCartesian fibration and a Cartesian fibration.

\begin{remark}\label{r1}
Let $0\leq i\leq j\leq n$.  
Consider the morphism $c_1\xra{\lag i \leq j \rag} [n]$.
The coCartesian monodromy functor of~(\ref{e7}) over the morphism is the unique colimit preserving functor $\iota_!\colon \PShv(\bTheta_i) \to \PShv(\bTheta_j)$ extending the composite functor $\bTheta_i \hookrightarrow \bTheta_j \hookrightarrow \PShv(\bTheta_j)$.  
The Cartesian monodromy functor of~(\ref{e7}) over this same morphism is the functor $\PShv(\bTheta_i)\la \PShv(\bTheta_j)\colon \iota^\ast$ given by pullback along $\bTheta_i \hookrightarrow \bTheta_j$.  
Notice that both of these monodromy functors preserve colimits, and that the coCartesian monodromy functor is left adjoint to the Cartesian monodromy functor.  

\end{remark}

Consider the $\infty$-category of sections of~(\ref{e7}):
\[
\Gamma\bigl(\PShv(\bTheta_\bullet)\bigr) 
~:=~
\Fun_{/[n]}\bigl( [n] , \PShv(\bTheta_\bullet) \bigr)~.
\]
Explicitly, an object of $\Gamma\bigl(\PShv(\bTheta_\bullet)\bigr)$ is, for each $0\leq i\leq n$, a presheaf $\cF_i\in \PShv(\bTheta_i)$, together with, for each $0<i\leq n$, a morphism $\iota_!\cF_{i-1}\to \cF_i$ between presheaves on $\bTheta_i$.  
Now, because $n\in [n]$ is a final object, Cartesian monodromy of the unique morphisms in $[n]$ to this final object define a functor from the fiber over $n$ to this $\infty$-category of sections:
\begin{equation}\label{e8}
\PShv(\bTheta_n)
\longrightarrow
\Gamma\bigl(\PShv(\bTheta_\bullet)\bigr) ~.
\end{equation}
This functor~(\ref{e8}) is fully faithful, and its image consists of the \emph{Cartesian} sections, which are those sections that carry morphisms to~(\ref{e7})-Cartesian morphisms.  
Precomposing with the Yonda functor $\bTheta_n \hookrightarrow \PShv(\bTheta_n)$ determines the solid diagram among $\infty$-categories:
\[
\xymatrix{
\bTheta_n  \ar[rr]^-{\sf ff}_-{\rm Yoneda}  \ar[d]_-{\sf ff}
&&
\PShv(\bTheta_n)  \ar[rr]^-{\sf ff}_-{(\ref{e8})}    \ar@{-->}[dll]^-{\sf id}  
&&
\Gamma\bigl(\PShv(\bTheta_\bullet)\bigr)   \ar@{-->}@(d,-)[dllll]^-{\sfN}
\\
\PShv(\bTheta_n)  
&&
&&
}
\]
Left Kan extensions define the fillers in this diagram, which is indeed a commutative diagram because each of the solid arrows is a fully faithful functor.
From the universal property of the Yoneda functor, the inner filler is the identity functor on $\PShv(\bTheta_n)$, as indicated.
As is always the case for left Kan extensions through a Yoneda functor, the outer filler is the restricted Yoneda functor.
(We give this left Kan extension the same notation as~(\ref{0}) because it extends that functor, as we will see.)
From the universal property of left Kan extensions, the resulting triangle among presentable $\infty$-categories
\[
\xymatrix{
&&&
&&&
\Gamma\bigl(\PShv(\bTheta_\bullet)\bigr)     \ar@(r,r)[dd]^-{\sf id}   \ar[dll]^-{\sfN}
\\
\bTheta_n    \ar[rr]      \ar@(d,-)[drrrrrr]  \ar@(u,-)[urrrrrr]  
&&
\PShv(\bTheta_n)  \ar[rr]^-{\sf id}    \ar@(d,-)[drrrr]^-{(\ref{e8})}    \ar@(u,-)[urrrr]_-{(\ref{e8})}    
&&
\PShv(\bTheta_n)      \ar[drr]^-{(\ref{e8})}  
&&
\Rightarrow
\\
&&&
&&&
\Gamma\bigl(\PShv(\bTheta_\bullet)\bigr) 
}
\]
lax-commutes, as indicated.
By construction, the resulting outer lax-commutative triangle among $\infty$-categories is, in fact, a commutative triangle.  
From the universal property of the Yoneda functor $\bTheta_n \to \PShv(\bTheta_n)$ as a colimit completion, it follows that the second-to-outer lax-commutative triangle is also, in fact, a commutative triangle.
This concludes the construction of an adjunction
\begin{equation}\label{e10}
(\ref{e8})\colon 
\PShv(\bTheta_n)  
~\rightleftarrows~
\Gamma\bigl(\PShv(\bTheta_\bullet)\bigr)  \colon \sfN   ~.
\end{equation}
Explicitly, the left adjoint evaluates on a presheaf $\cF\in \PShv(\bTheta_n)$ as the section $i\mapsto \cF_{|\bTheta_i^{\op}}$; in the case that $\cF$ is represented by an object $T\in \bTheta_n$, we implement the other notation $T_{\leq \bullet}$ in place of $\cF_{|\bTheta_\bullet^{\op}}$. 
Explicitly, the right adjoint evaluates on a section $\cF_\bullet$ as the presheaf $T\mapsto \Map_{\Gamma}(T_{\leq \bullet} , \cF_\bullet)$, whose values are spaces of morphisms in the $\infty$-category $\Gamma\bigl( \PShv(\bTheta_\bullet)\bigr)$.

Inspecting the definition of the restricted Yoneda functor $\sfN\colon \Gamma\bigl(\PShv(\bTheta_\bullet)\bigr)  \to \PShv(\bTheta_n)$, as $n$ varies, reveals the following.
\begin{observation}\label{t3}
For each $0\leq i\leq j\leq n$, the diagram among $\infty$-categories
\[
\xymatrix{
\PShv(\bTheta_j)  \ar[d]
&&
\Gamma\bigl(\PShv(\bTheta_{\bullet \leq j})\bigr)   \ar[d]    \ar[ll]_-{\sfN}
\\
\PShv(\bTheta_i) 
&&
\Gamma\bigl(\PShv(\bTheta_{\bullet \leq i})\bigr)    \ar[ll]_-{\sfN}
}
\]
canonically commutes.

\end{observation}

\subsection{Restricting the adjunction}
We now show that the adjunction~(\ref{e10}) restricts to Segal objects.

Recall the $\infty$-category of~(\ref{e7}).
Consider the full $\infty$-subcategories
\begin{equation}\label{e11}
\PShv(\bTheta_\bullet)
~\supset~
\Shv(\bTheta_\bullet)
~\supset~
\Shv^{\sf unv}(\bTheta_\bullet)
\end{equation}
consisting of those pairs $\bigl(i\in [n] , \cF\in \PShv(\bTheta_i)\bigr)$ for which $\cF_i\in \Shv(\bTheta_i)$, and for which $\cF_i\in \Shv^{\sf unv}(\bTheta_i)$, respectively.

\begin{lemma}\label{t6}
In the commutative diagram among $\infty$-categories,
\[
\xymatrix{
\PShv(\bTheta_\bullet)  \ar[drr]  
&&
\Shv(\bTheta_\bullet)  \ar[d]    \ar[ll]
&&
\Shv^{\sf unv}(\bTheta_\bullet)  \ar[dll]     \ar[ll]
\\
&&
[n]
&&    ,
}
\]
each of the vertical functors is both a coCartesian fibration and a Cartesian fibration, and each of the horizontal functors is fully faithful and preserves coCartesian morphisms and Cartesian morphisms over $[n]$.  
\end{lemma}

\begin{proof}
By Definition~\ref{def.segal.cov}, both Segal covers and univalence diagrams are, in particular, limit diagrams in $\bTheta_n$.  
By direct inspection, for each $0\leq i\leq j \leq n$, the fully faithful functor $\bTheta_i \hookrightarrow \bTheta_j$ carries Segal covers to Segal covers and carries univalence diagrams to univalence diagrams. 
From these two points, it follows that the adjunction $\iota_!\colon \PShv(\bTheta_i) \rightleftarrows \PShv(\bTheta_j)\colon \iota^\ast$ restricts as an adjunction
\[
\iota_!\colon \Shv(\bTheta_i)
\rightleftarrows
\Shv(\bTheta_j) \colon \iota^\ast ~,
\]
which further restricts as an adjunction
\[
\iota_!\colon \Shv^{\sf unv}(\bTheta_i)
\rightleftarrows
\Shv^{\sf unv}(\bTheta_j) \colon \iota^\ast ~,
\]
The lemma follows. 
\end{proof}

For each $0\leq i \leq n$, via Bousfield localization, the fully faithful inclusion between presentable $\infty$-categories, $\Shv^{\sf unv}(\bTheta_i) \hookrightarrow \Shv(\bTheta_i)$, is a right adjoint:
\begin{equation}\label{e20}
(-)^{\widehat{~}}_{\sf unv}\colon \Shv(\bTheta_i) ~\rightleftarrows~ \Shv^{\sf unv}(\bTheta_i)~.
\end{equation}
The left adjoint is \emph{univalent-completion}.  
\begin{lemma}\label{t7}
The fully faithful inclusion 
\[
\Shv(\bTheta_\bullet)
\hookleftarrow 
\Shv^{\sf unv}(\bTheta_\bullet)
\]
is a right adjoint functor.  
Its left adjoint functor lies over $[n]$, and the adjunction is given on fibers over $i\in [n]$ as the Bousfield localization $(-)^{\widehat{~}}_{\sf unv}  \colon \Shv(\bTheta_i)
\rightleftarrows
\Shv^{\sf unv}(\bTheta_i)$ implementing univalent-completion.  
\end{lemma}

\begin{proof}
The existence of a left adjoint is a condition on the given fully faithful functor.  
Because the given fully faithful functor is a coCartesian functor between coCartesian fibrations over $[n]$, the result is proved upon showing that the following two points.
\begin{enumerate}
\item
For each $0\leq i\leq n$, the functor between fibers over $i$ is a right adjoint in an adjunction,
\[
\sL_i\colon \Shv(\bTheta_i)
\rightleftarrows
\Shv^{\sf unv}(\bTheta_i)~.
\]

\item
For each $0<i \leq n$, the diagram
\[
\xymatrix{
\Shv(\bTheta_{i-1})    \ar@(u,u)[rr]^-{\sL_{i-1}}  \ar[d]
&
\Downarrow
&
\Shv^{\sf unv}(\bTheta_{i-1})  \ar[d]
\\
\Shv(\bTheta_i)    \ar[rr]^-{\sL_i}  
&&
\Shv^{\sf unv}(\bTheta_i)  
}
\]
lax-commutes.  

\end{enumerate}
The first point is exactly the adjunction~(\ref{e20}), in which $\sL_i = (-)^{\widehat{~}}_{\sf unv}$ is univalent-completion.  
The fully faithful functor $\Shv^{\sf unv}(\bTheta_{i-1})  \hookrightarrow \Shv^{\sf unv}(\bTheta_i)$ is a left adjoint, with right adjoint given by restriction: $\cC\mapsto \cC_{<i}$.
Therefore, the sought lax-commutative diagram is implemented from the commutativity of the diagram involving right adjoints to the sought lax-commutative diagram:
\[
\xymatrix{
\Shv(\bTheta_{i-1})    
&
&
\Shv^{\sf unv}(\bTheta_{i-1})  \ar[ll]
\\
\Shv(\bTheta_i)    \ar[u]
&&
\Shv^{\sf unv}(\bTheta_i)    \ar[ll]  \ar[u]
}
\]
\end{proof}

Taking sections, Lemma~\ref{t7} has the following useful consequence.
\begin{cor}\label{t5}
The fully faithful inclusion 
\[
\Gamma\bigl(\Shv(\bTheta_\bullet)\bigr)  
\hookleftarrow 
\Gamma\bigl(\Shv^{\sf unv}(\bTheta_\bullet)\bigr)
\]
is a right adjoint functor.  
Its left adjoint carries a section $\cF_\bullet$ to the section $(\cF_\bullet)^{\widehat{~}}_{\sf unv}$, whose value on $i\in [n]$ is the univalent-completion of the Segal sheaf $\cF_i\in \Shv(\bTheta_i)$.  

\end{cor}

\begin{remark}\label{r8}
Verifying that a presheaf on $\bTheta_n$ satisfies the Segal condition, as defined in Definition~\ref{def.segal.cov}, can be reduced to a simpler problem, as we now explain.
Each closed morphism in $\bTheta_n$ is a monomorphism.
Therefore, for each $T\in \bTheta_n^{\sf cls}$, the overcategory $\bTheta^{\sf cls}_{n/T}$ is a a poset.
Inspecting the definition of the category $\bTheta_n$, this poset $\bTheta^{\sf cls}_{n/T}$ is, in fact, finite.  
Therefore, each colimit in $\bTheta_n^{\sf cls}$ can be expressed as a finite iteration of pushouts.  
It follows that a presheaf $\cF\in \PShv(\bTheta_n)$ is \emph{Segal} if and only if it carries (the opposites of) pushout diagrams in $\bTheta_n^{\sf cls}$ to pullback diagrams among spaces.
We make implicit use of this reduction as we proceed.
\end{remark}

\begin{lemma}\label{t4}
The adjunction~(\ref{e10}) restricts as an adjunction
\begin{equation}\label{e12}
(\ref{e8})\colon 
\Shv(\bTheta_n)  
~\rightleftarrows~
\Gamma\bigl(\Shv(\bTheta_\bullet)\bigr)  \colon \sfN   ~.
\end{equation}

\end{lemma}

\begin{proof}
Because the asserted restriction is to fully faithful $\infty$-subcategories in the adjunction~(\ref{e10}), we need only show that the left and the right adjoint functors restrict as desired.

From its definition, Lemma~\ref{t6} gives that the functor~(\ref{e8}) restricts as a functor
\[
(\ref{e8})\colon 
\Shv(\bTheta_n)  
~\rightleftarrows~
\Gamma\bigl(\Shv(\bTheta_\bullet)\bigr) ~,
\]
which is necessarily fully faithful. 

It remains to prove that $\sfN$ restricts likewise. 
Let $\cF_\bullet\in \Gamma\bigl(\Shv(\bTheta_\bullet)\bigr)$.  
We must show, then, that the presheaf $\sfN(\cF_\bullet)\colon \bTheta_n^{\op} \to \Spaces$ carries (opposites) of Segal covers to limit diagrams.  
By definition of $\sfN$ as a restricted Yoneda functor, this is implied by the functor $\bTheta_n \to \fCat_n$ carrying Segal covers to colimit diagrams.  
By definition of a Segal cover, this is implied by the functor $\bTheta_n^{\sf cls} \to \fCat_n$ preserving colimit diagrams.
By definition of $\fCat_n$ as a full $\infty$-subcategory of $\Fun([n],\Cat_n)$, this is implied by each of the forgetful functors $(-)_{\leq i}\colon \bTheta_n^{\sf cls} \to \Cat_n \xra{\rm forget}\Cat_i$ preserving colimit diagrams.

Now let $T_\bullet \colon \cJ^{\tr} \to \bTheta_n^{\cls}$ be a colimit diagram.
Denote the value on the cone point as $T:=T_\ast \in \bTheta_n$.
We must show that the composite functor $(T_\bullet)_{\leq i}\colon \cJ^{\tr} \to \Cat_i$ is a colimit diagram.  
In general, for each $k$, the full subcategory of $\bTheta^{\cls}_k$ consisting of the cells strongly generates; also, by definition, the functor $\bTheta_k \hookrightarrow \Cat_k$ preserves such colimit diagrams.
We can therefore reduced to the case that the functor $\cJ \simeq \cE(T) \hookrightarrow \bTheta^{\cls}_{n/T}$ is the inclusion of the full subcategory $\cE(T)\subset \bTheta^{\cls}_{n/T}$ consisting of those closed morphisms $T_j\to T$ for which $T_j \simeq c_k$ is a $k$-cell for some $0\leq k\leq n$.
This, a priori, $\infty$-category $\cE(T)$ is in fact a finite poset (see Remark~\ref{r8}).
The lemma is proved once we establish the following sequence of equivalences among $(\infty,i)$-categories:
\begin{eqnarray}
\nonumber
\underset{C\in \cE(T)}\colim  C_{\leq i}
&
\xla{~\simeq~}
&
\nonumber
\underset{C \in \cE(T)}    \colim     ~ \underset{C' \in \cE(C)_{\leq i}}   \colim   C'
\\
\nonumber
&
\xla{~\simeq~}
&
\nonumber
\underset{C' \in \cE(T)_{\leq i} } \colim   C'
\\
\nonumber
&
\xra{~\simeq~}
&
\nonumber
T_{\leq i} ~.
\end{eqnarray}

From the universal property of the right adjoint functor $(-)_{\leq i}\colon \Cat_n \to \Cat_i$, for each $S\in \bTheta_n$, there is a canonical identification as a colimit:
\[
S_{\leq i}~\simeq~\colim\bigl(\bTheta_{i/S} \to \bTheta_i \hookrightarrow \Cat_i \bigr)~.
\]
Consider the subcategory $\bTheta^{\cls}_{i/S}\subset \bTheta_{i/S}$ consisting of the closed morphisms to $S$ and closed morphisms among them.
Via the active-closed factorization system on the category $\bTheta_i$, the inclusion of this subcategory is a final functor.  
Consider the full subcategory $\cE(S)_{\leq i}\subset  \cE(S) \subset  \bTheta^{\cls}_{i/S}$ consisting of those $(C\to S)$ for which $C \simeq c_k$ is a $k$-cell for some $0\leq k\leq i$.
Now, for each $k$, the full subcategory of $\bTheta^{\cls}_k$ consisting of the cells strongly generates; also, by definition, the functor $\bTheta_k \hookrightarrow \Cat_k$ preserves such colimit diagrams.
We conclude from these observations an identification
\[
S_{\leq i}~\simeq~\colim\bigl( \cE(S)_{\leq i} \hookrightarrow \bTheta^{\cls}_{i/S} \hookrightarrow \bTheta_{i/S} \to \bTheta_i \hookrightarrow  \Cat_i\bigr)~.
\]
Applying this to $S=T$ gives the final equivalence in the above string.  
Notice that the assignments $S\mapsto \cE(S)$ and $S\mapsto \cE(S)_{\leq i}$ each evidently extend as a functors $\bTheta^{\cls}_n\to {\sf Poset}_{/\bTheta^{\cls}_n}$.
This gives the first of the equivalences.
The second equivalence follows from Quillen's Theorem A, using the following observation.
\begin{itemize}
\item[]
Let $C' \to T$ be a closed morphism in $\bTheta_n$ from an $i$-cell.
Then the poset of factorizations of this closed morphism through a closed morphism $C \to T$ from a cell, has an initial object.  

\end{itemize}
Finally, the composition of this string of equivalence is evidently the canonical morphism we intended to show is an equivalence.  
\end{proof}

Concatenating Corollary~\ref{t5} and Lemma~\ref{t4} results in the composite adjunction
\begin{equation}\label{e14}
({-}_{|\bTheta_\bullet^{\op}})^{\widehat{~}}_{\sf unv}\colon 
\Shv(\bTheta_n)  
~ \rightleftarrows ~
\Gamma\bigl(\Shv(\bTheta_\bullet)\bigr)  
~ \rightleftarrows ~
\Gamma\bigl(\Shv^{\sf unv}(\bTheta_\bullet)\bigr)  \colon \sfN  ~.
\end{equation}

Unwinding definitions reveals the next observation.
\begin{observation}\label{t8}
There is a canonical fully faithful functor between $\infty$-categories:
\begin{equation}\label{e15}
\fCat_n  
~\hookrightarrow~
\Gamma\bigl(\Shv^{\sf unv}(\bTheta_\bullet)\bigr)    ~.
\end{equation}
The image consists of those sections $\cC_\bullet = (\cC_0 \to \cC_1 \to \dots \to \cC_n)$ for which, for each $0\leq i\leq j\leq n$, the functor $\cC_i \to \cC_j$ between $(\infty,j)$-categories is $i$-connective.  

\end{observation}

The proof of the next result occupies~\S\ref{sec.in.fCat}.
\begin{lemma}\label{t9}
The value of the left adjoint of the adjunction~(\ref{e14}) on a Segal sheaf $\cF\in \Shv(\bTheta_n)$ lies in the image of the fully faithful functor of Observation~\ref{t8}:
\[
(\cF_{|\bTheta_\bullet^{\op}})^{\widehat{~}}_{\sf unv}~\in~\fCat_n~.  
\]
\end{lemma}

Through Observation~\ref{t8}, Lemma~\ref{t9} has the following consequence.
\begin{cor}\label{t10}
The adjunction~(\ref{e14}) restricts as an adjunction
\begin{equation}\label{e015}
({-}_{|\bTheta_\bullet^{\op}})^{\widehat{~}}_{\sf unv}\colon
\Shv(\bTheta_n)
~\rightleftarrows~
\fCat_n  \colon \sfN ~.
\end{equation}
\end{cor}

\subsection{Explicating the adjunction}
After Corollary~\ref{t10}, our main result (Theorem~\ref{main.result.1}) is implied by showing that both the unit and the counit transformations of the adjunction~(\ref{e015}) are equivalences.  
So we explicate the values of left and right adjoints, as well as the unit and the counit, of the adjunction~(\ref{e015}).

\subsubsection{\bf The left adjoint}
The value of the left adjoint $(-_{|\bTheta_\bullet^{\op}})^{\widehat{~}}_{\sf unv}$ of the adjunction~(\ref{e14}) on a Segal sheaf $\cF\in \Shv(\bTheta_n)$ is the section $(\cF_{|\bTheta_\bullet^{\op}})^{\widehat{~}}_{\sf unv}$ of the functor $\Shv^{\sf unv}(\bTheta_\bullet) \to [n]$ that is the assignment
\[
[n]\ni i
\mapsto
(\cF_{|\bTheta_i^{\op}})^{\widehat{~}}_{\sf unv}\in \Shv^{\sf unv}(\bTheta_i)~,
\]
which is the univalent completion of the restriction $\cF_{|\bTheta_i^{\op}}\in \Shv(\bTheta_i)$.

\subsubsection{\bf The right adjoint}\label{sec.right.adjoint}
The value of the right adjoint $\sfN$ of the adjunction~(\ref{e14}) on a section $\cC_\bullet \in \Gamma\bigl(\Shv^{\sf unv}(\bTheta_\bullet)\bigr)$ is the Segal sheaf $\sfN(\cC_\bullet)\in \Shv(\bTheta_n)$ that is the assignment
\[
\bTheta_n^{\op} \ni T  
\mapsto 
\Map_{\Gamma_{\leq n}}(T_{\leq \bullet} , \cC_\bullet) \in \Spaces~,
\]
which is the space of morphisms in $\Gamma\bigl(\Shv^{\sf unv}(\bTheta_\bullet)\bigr)$ from $T_{\leq \bullet} = (T_{\leq 0} \to T_{\leq 1}\to \dots \to T_{\leq n-1} \to T)$ to $\cC_\bullet$.  
The next result (Corollary~\ref{t1}) makes the values of the Segal sheaf $\sfN(\cC_\bullet)$ more explicit.

\begin{observation}\label{mapping.spaces}
For $\cC_\bullet$ and $\cD_\bullet$ sections of the functor $\PShv(\bTheta_\bullet) \to [n]$, the canonical square among spaces of morphisms
\[
\xymatrix{
\Map_{\Gamma_{\leq n}}(\cC_\bullet,\cD_\bullet)  \ar[rr]  \ar[d]
&&
\Map_{\PShv(\bTheta_n)}(\cC_n,\cD_n)  \ar[d]
\\
\Map_{\Gamma_{<n}}(\cC_{\bullet <  n} , \cD_{\bullet < n})  \ar[rr]
&&
\Map_{\PShv(\bTheta_n)}(\cC_{n-1},\cD_n)
}
\]
is a pullback.

\end{observation}

\begin{proof}
The coCartesian monodromy functors of the coCartesian fibration $\PShv(\bTheta_\bullet)\to [n]$ are given by left Kan extensions along fully faithful functors.
Therefore, these coCartesian monodromy functors are fully faithful.  
These coCartesian monodromy functors thusly define a fully faithful functor to the fiber over the final object $n\in [n]$:
\[
\Gamma\bigl(\PShv(\bTheta_\bullet)\bigr)
~\hookrightarrow~
\Fun\bigl([n] , \PShv(\bTheta_n)\bigr) =: \PShv(\bTheta_n)^{[n]}~.
\]

The canonical functor $[n-1]\underset{\{n-1\}}\amalg \{n-1 < n\} \to [n]$ between $\infty$-categories is an equivalence from the pushout.  
Consequently, for $\cX$ an $\infty$-category, and for $x_\bullet , y_\bullet \in \Fun([n],\cX)=: \cX^{[n]}$ two functors $[n]\to \cX$, the canonical square among spaces of morphisms
\[
\xymatrix{
\cX^{[n]}(x_{\bullet},y_{\bullet})  \ar[r]  \ar[d]
&
\cX(x_n,y_n)  \ar[d]
\\
\cX^{[n-1]}(x_{\bullet<n} ,y_{\bullet<n})  \ar[r]
&
\cX(x_{n-1} , y_n)
}
\]
is a pullback.
Apply this to the case $\cX = \PShv(\bTheta_n)$.
\end{proof}

\begin{cor}\label{t1}
Let $\cC_\bullet$ be a section of the functor $\PShv(\bTheta_\bullet) \to [n]$.
Let $T\in \bTheta_n$.
There is a canonical pullback diagram among spaces:
\[
\xymatrix{
\sfN(\cC_\bullet)(T)  \ar[rr]  \ar[d]
&&
\cC_n(T)  \ar[d]
\\
\sfN(\cC_{\bullet<n})(T_{<n})  \ar[rr]
&&
\cC_n(T_{<n})~.
}
\]
Alternatively, there is a canonical limit diagram among spaces:
\[
\xymatrix{
&&
\sfN(\cC_\bullet)(T)  \ar[dll]  \ar[dl]    \ar[d]  \ar[dr]  \ar[drr]
&&
\\
\cC_0(T_{\leq 0})  \ar[dr]
&
\cC_1(T_{\leq 1})   \ar[dr]  \ar[d]
&
\cdots  \ar[dr]  \ar[d]
&
\cC_{n-1}(T_{\leq n-1})  \ar[dr]  \ar[d]
&
\cC_n(T)      \ar[d]
\\
&
\cC_1(T_{\leq 0})  
&
\cC_2(T_{\leq 1})
&
\cC_{n-1}(T_{\leq n-2})
&
\cC_n(T_{\leq n-1})
.
}
\]
\end{cor}

Corollary~\ref{t1} makes apparent the following.
\begin{observation}\label{t14}
Let $0\leq i \leq n$.
Let $\cF_\bullet$ be a section of the functor $\PShv(\bTheta_\bullet) \to [n]$. 
Restricting this section over $[i] = \{0<\dots<i\}\subset[n]$ determines the section $\cF_{\bullet \leq i}$ of $\PShv(\bTheta_\bullet) \to [i]$.
There is a canonical equivalence between presheaves on $\bTheta_i$:
\[
\sfN(\cF_{\bullet\leq i})
~\simeq~
\sfN(\cF_\bullet)_{|\bTheta_i^{\op}}
~.
\]
\end{observation}

\subsubsection{\bf The unit}\label{sec.unit}
The unit of the adjunction~(\ref{e14}) evaluates on each Segal sheaf on $\bTheta_n$ as the morphism between presheaves
\[
{\rm unit}\colon \cF
\longrightarrow
\sfN\bigl( (\cF_{|\bTheta_\bullet^{\op}})^{\widehat{~}}_{\sf unv}\bigr)
\]
whose value on $T\in \bTheta_n$ is described inductively (via Observation~\ref{t14}) through Corollary~\ref{t1} as the square among spaces
\begin{equation}\label{e21}
\xymatrix{
\cF(T)  \ar[rrrr]  \ar[d]
&&&&
\cF^{\widehat{~}}_{\sf unv}(T)  \ar[d]
\\
\cF(T_{<n})  \ar[rr]^-{\rm induction}
&&
\sfN\bigl((\cF_{|\bTheta_{\bullet<n}})^{\widehat{~}}_{\sf unv}  \bigr)  \ar[rr]
&&
\cF^{\widehat{~}}_{\sf unv}(T_{<n})~.
}
\end{equation}
Without compressing this description via induction, the unit transformation can be described through Corollary~\ref{t1} as the canonical diagram among spaces
\begin{equation}\label{e22}
\xymatrix{
&&
\cF(T)  \ar[dll]  \ar[dl]    \ar[d]  \ar[dr]  
&&
\\
(\cF_{|\bTheta_0^{\op}})^{\widehat{~}}_{\sf unv} (T_{\leq 0})  \ar[dr]
&
(\cF_{|\bTheta_1^{\op}})^{\widehat{~}}_{\sf unv} (T_{\leq 1})   \ar[dr]  \ar[d]
&
(\cF_{|\bTheta_{n-1}^{\op}})^{\widehat{~}}_{\sf unv} (T_{\leq n-1})  \ar[dr]  \ar[d]
&
(\cF_{|\bTheta_n^{\op}})^{\widehat{~}}_{\sf unv} (T)      \ar[d]
\\
&
(\cF_{|\bTheta_1^{\op}})^{\widehat{~}}_{\sf unv} (T_{\leq 0})  
&
\dots
&
(\cF_{|\bTheta_n^{\op}})^{\widehat{~}}_{\sf unv} (T_{\leq n-1})
.
}
\end{equation}

The proof of the next result occupies~\S\ref{sec.unit}.
\begin{lemma}\label{t11}
For each Segal sheaf $\cF$ on $\bTheta_n$, and for each $T\in \bTheta_n$, both of the diagrams among spaces~(\ref{e21}) and~(\ref{e22}) are limit diagrams.

\end{lemma}

\subsubsection{\bf The counit}\label{sec.counit}
The counit of the adjunction~(\ref{e14}) evaluates on each section $\cC_\bullet$ of the functor $\Shv^{\sf unv}(\bTheta_\bullet^{\op}) \to [n]$ as the morphism between sections
\[
{\rm counit}\colon 
(\sfN(\cC_\bullet)_{|\bTheta_\bullet^{\op}})^{\widehat{~}}_{\sf unv}  \longrightarrow  \cC_\bullet
\]
whose value on $i\in [n]$ is described, through Observation~\ref{t14}, as the canonical functor between $(\infty,i)$-categories
\[
\bigl( \sfN(\cC_{\bullet\leq i}) \bigr)^{\widehat{~}}_{\sf unv}  \longrightarrow  \cC_i
\]
from the univalent-completion of the Segal sheaf $\sfN(\cC_{\bullet\leq i})$ on $\bTheta_i$ that evaluates on each $T\in \bTheta_i$ as the canonical map
\begin{equation}\label{e34}
\sfN(\cC_{\bullet\leq i})(T)
\longrightarrow
\cC_i(T) 
\end{equation}
as in Corollary~\ref{t1}.

\begin{observation}\label{t24}
Let $\un{\cC}$ be a flagged $(\infty,n)$-category.
Through the fully faithful functor~(\ref{e15}), regard $\un{\cC}$ as a section of the functor $\Shv^{\sf unv}(\bTheta_\bullet) \to [n]$.
The connectivity assumptions on each $\cC_i\to \cC_j$ ensure that, for each $0\leq i\leq n$ and each $T\in \bTheta_i$, the map~(\ref{e34}) between spaces
is surjective (on path components).  
From the 2-out-of-3 property for surjections, it follows that the counit evaluates as a surjection
\[
{\rm counit}\colon 
\bigl(  \sfN(\cC_{\bullet \leq i}) \bigr)^{\widehat{~}}_{\sf unv}(T)  \longrightarrow  \cC_i(T)
\]
is surjective (on path components).  

\end{observation}

The proof of the next result occupies~\S\ref{sec.counit}.
\begin{lemma}\label{t12}
Let $\un{\cC} \in \fCat_n$ be a flagged $(\infty,n)$-category.  
For each $0\leq i\leq n$, the canonical morphism between Segal sheaves on $\bTheta_i$
\[
{\rm counit}\colon 
\sfN(\cC_{\bullet\leq i}) \longrightarrow  \cC_i
\]
witnesses a univalent-completion: $(\sfN(\cC_{\bullet\leq i}))^{\widehat{~}}_{\sf unv} \xra{~\simeq~} \cC_i$.

\end{lemma}

\section{Univalence completion}
We establish a formula for univalent-completion.  
Section~\S\ref{sec.formula} establishes the formula for univalent-completion.
Section~\S\ref{sec.correct} justifies that this formula indeed implements univalent-completion.

\subsection{The formula}\label{sec.formula}
We establish the somewhat explicit formula for univalent-completion.  

Notation~\ref{d3} is an instance of the following.
\begin{definition}\label{d1}
The functor
\begin{equation}\label{e1}
\sE\colon \bTheta_n
\longrightarrow 
\PShv(\bTheta_n)
\end{equation}
carries an object $T=[p](T_1,\dots,T_p)\in \bTheta_n$ to the presheaf of sets
\[
\sE(T)\colon 
S=[q](S_1,\dots,S_q)  
\mapsto
\Bigl\{ \Bigl(
\{0,\dots,q\}\xra{f} \{0,\dots,p\}
~,~
\bigl(    g_{ij} \in \sE(T_i)(S_j)    \bigr)_{0<i\leq p; ~j\in {\sf Hull}\{ f(i-1),f(i)\}}
\Bigr) \Bigr\}    ~.
\]

\end{definition}

Each value of the functor~(\ref{e1}) is an explicit description of strict $n$-groupoid-completion, as we now observe.

\begin{observation}\label{t2}
The functor~(\ref{e1}) is identical to the composite functor
\[
\sE\colon \bTheta_n \hookrightarrow \Cat_n^{\sf strict} \xra{\text{$n$-groupoid completion}} \Gpd_n^{\sf strict}
\hookrightarrow 
\Cat_n^{\sf strict}
\xra{~\rm nerve~}
\PShv(\bTheta_n)~,
\]
which we briefly explain.
Here, $\Gpd_n^{\sf strict}\subset \Cat_n^{\sf strict}$ is its full $(2,1)$-subcategory consisting of those categories enriched in $\Gpd_{n-1}^{\sf strict}$ for which each (1-)morphism is invertible. 
The first functor is the fully faithful functor of Remark~\ref{r7}.
The second functor is the left adjoint to the inclusion of the strict $n$-groupoids. 
The third functor is the inclusion of the strict $n$-groupoids.
The fourth functor is the \emph{nerve} functor of Remark~\ref{r7}.

\end{observation}

Observation~\ref{t2} yields the following.
\begin{observation}\label{t22}
The functor~(\ref{e1}) factors through $n$-groupoids, and in particular Segal sheaves: $\sE\colon \bTheta_n \xra{(\ref{e1})} \nGpd\underset{\rm Obs~\ref{t19}}\subset \Shv(\bTheta_n)$.

\end{observation}

Using the adjoint functor theorem, there is a Bousfield localization 
\begin{equation}\label{segal.cplt}
\Shv(\bTheta_n) ~\rightleftarrows~\Shv^{\sf unv}(\bTheta_n)~,\qquad \cF\mapsto \cF^{\widehat{~}}_{\sf unv}~.
\end{equation}
We now give a somewhat explicit formula for this left adjoint, which is inspired the \emph{classifying diagram} construction of~\cite{rezk}, which is elaborated in~\cite{joyaltierney}.
Consider the composite functor
\begin{equation}\label{e19}
L  \colon \PShv(\bTheta_n)
\xra{ \cF \mapsto \cF(-\times \sE(\bullet) )  }
\PShv(\bTheta_n \times \bTheta_n)
\xra{ \cH \mapsto |\cH|_\bullet  }
\PShv(\bTheta_n)
~,\qquad
\cF\mapsto L(\cF) := \bigl( T\mapsto \bigl| \cF(T\times \sE(\bullet) ) |_{\bullet} \bigr)~.
\end{equation}
The initial functor $\ast \xra{\lag [0]\rag} \bTheta_n^{\op}$ determines the initial functor $\bTheta_n^{\op} = \bTheta_n^{\op}\times \ast \to \bTheta_n^{\op}\times \bTheta_n^{\op}$.
This, in turn, determines the natural transformation
\begin{equation}\label{e24}
\id 
\longrightarrow
L
~,\qquad
\cF\mapsto \Bigl( \cF =  \cF( - \times \sE([0]) ) 
\to   
\bigl| \cF(- \times \sE(\bullet) ) |_{\bullet} =: L(\cF)  \Bigr)  ~.
\end{equation}

The proof of the next result occupies~\S\ref{sec.correct}.
\begin{prop}\label{completion}
There is a canonical commutative diagram among $\infty$-categories:
\[
\xymatrix{
\Shv(\bTheta_n)  \ar[rr]^-{(-)^{\widehat{~}}_{\sf unv}}   \ar[d]
&&
\Shv^{\sf unv}(\bTheta_n)  \ar[d]
\\
\PShv(\bTheta_n)  \ar[rr]^-{L}
&&
\PShv(\bTheta_n)   ,
}
\]
which extends as a commutative diagram among $\infty$-categories:
\[
\xymatrix{
c_1  \times  \Shv(\bTheta_n)  \ar[rr]^-{\rm unit}  \ar[d]
&&
\Shv(\bTheta_n)  \ar[d]
\\
c_1  \times  \PShv(\bTheta_n)  \ar[rr]^-{(\ref{e24})}
&&
\PShv(\bTheta_n)    .
}
\]
In other words, for each Segal sheaf $\cF\in \PShv(\bTheta_n)$, and for each $T\in \bTheta_n$, there is a canonical identification between spaces:
\[
\bigl|   \Map\bigl( T \times \sE(\bullet) , \cF \bigr)   \bigr|
\xra{~\simeq~} \cF^{\widehat{~}}_{\sf unv}(T)   ~;
\]
with respect to this identification, the value of the unit of the adjunction~(\ref{segal.cplt}) on $\cF$ evaluates on $T \in \bTheta_n$ as the canonical map between spaces
\[
\cF(T)\simeq \Map\bigl(T \times \sE([0]) , \cF \bigr)  \longrightarrow   \bigl| \Map\bigl( T\times \sE(\bullet) , \cF \bigr) \bigr|\simeq \cF^{\widehat{~}}_{\sf unv}(T)   ~.
\]
\end{prop}

\subsection{The formula is correct}\label{sec.correct}
We prove Proposition~\ref{completion}.
This result is proved upon establishing the following features of the functor $L$ of~(\ref{e19}) and the natural transformation $\id \to L$ of~(\ref{e24}).  
\begin{enumerate}

\item
Should $\cF$ carry (the opposites of) Segal covers to limit diagrams, then the presheaf $L(\cF)$ also carries (the opposites of) Segal covers to limit diagrams.
In other words, there is a factorization of the restriction:
\[
\xymatrix{
\Shv(\bTheta_n)  \ar@{-->}[rr]^-L  \ar[d]
&&
\Shv(\bTheta_n)  \ar[d]
\\
\PShv(\bTheta_n)  \ar[rr]^-L   
&&
\PShv(\bTheta_n)    .
}
\]

\item
Should $\cF$ carry (the opposites of) Segal covers to limit diagrams, then the presheaf $L(\cF)$ carries (the opposites of) univalence diagrams to limit diagrams.
In other words, there is a further factorization of the restriction:
\[
\xymatrix{
\Shv(\bTheta_n)  \ar@{-->}[rr]^-L  \ar[dr]_-L
&&
\Shv^{\sf unv}(\bTheta_n)  \ar[dl]
\\
&
\PShv(\bTheta_n)  
& .
}
\]

\item
Should $\cF$ carry (the opposites of) Segal covers and univalence diagrams to limit diagrams, then the natural transformation $\id \to L$ evaluates as an equivalence.
In other words, the lax-commutative diagram
\[
\xymatrix{
\Shv^{\sf unv}(\bTheta_n)  \ar@(u,u)[rr]^-{\id}  \ar[dr]
&
\Downarrow
&
\Shv^{\sf unv}(\bTheta_n)  
\\
&
\PShv(\bTheta_n)  \ar[ur]_-L   
&   
.
}
\]
is, in fact, a commutative diagram.

\end{enumerate}
\begin{itemize}

\item 
{\bf Proof of~(1):}
Let $\cF\in \Shv(\bTheta_n)$ be a Segal sheaf.  
Let 
\[
\xymatrix{
T^0  \ar[r]  \ar[d]
&
T^+  \ar[d]
\\
T^-  \ar[r]
&
T
}
\]
be a pushout diagram on $\bTheta_n^{\sf cls}$.  
This diagram determines the diagram in $\Fun\bigl(\bTheta_n , \PShv(\bTheta_n)\bigr)$:
\[
\xymatrix{
T^0  \times \sE(\bullet)  \ar[rr]  \ar[d]
&&
T^+\times \sE(\bullet)  \ar[d]
\\
T^-\times \sE(\bullet)  \ar[rr]
&&
T  \times \sE(\bullet) .
}
\]
Through Observation~\ref{t2}, using that the Segal condition is closed under products of presheaves, this is in fact a diagram in $\Fun\bigl(\bTheta_n , \Shv(\bTheta_n)\bigr) \subset \Fun\bigl(\bTheta_n , \PShv(\bTheta_n)\bigr)$.
Established in~\cite{rezk-n}, products in $\Shv(\bTheta_n)$ distribute over colimits.  
Therefore the above diagram is a pushout diagram in $\Fun\bigl(\bTheta_n , \Shv(\bTheta_n)\bigr)$.
Applying the Segal sheaf $\cF$ results in the pullback diagram in $\Fun\bigl(\bTheta_n^{\op} , \Spaces\bigr)$:
\begin{equation}\label{e41}
\xymatrix{
\cF\bigl( T  \times \sE(\bullet)  \bigr) \ar[rr]  \ar[d]
&&
\cF\bigl( T^+  \times \sE(\bullet)  \bigr)   \ar[d]
\\
\cF\bigl( T^-  \times \sE(\bullet)  \bigr)   \ar[rr]
&&
\cF\bigl( T^0  \times \sE(\bullet)  \bigr)    .
}
\end{equation}
Because it is the case for each $\sE(\bullet)$, this diagram is in fact a diagram in the full $\infty$-subcategory $\nGpd\subset \Fun(\bTheta_n^{\op},\Spaces)$ of Observation~\ref{t19}.
Because this full $\infty$-subcategory is closed under limits, the above diagram is a pullback diagram in $\nGpd$. 
The functor $|-|\colon \nGpd \to \Spaces$ carries the above diagram to a diagram among spaces
\begin{equation}\label{e40}
\xymatrix{
L(\cF)(T) := \bigl| \cF\bigl( T  \times \sE(\bullet)  \bigr) \bigr|   \ar[rr]  \ar[d]
&&
\bigl| \cF\bigl( T^+ \times \sE(\bullet)  \bigr) \bigr| =: L(\cF)(T^+)  \ar[d]
\\
L(\cF)(T^-) := \bigl| \cF\bigl( T^-  \times \sE(\bullet)  \bigr) \bigr| \ar[rr]
&&
\bigl| \cF\bigl( T^0  \times \sE(\bullet)  \bigr) \bigr|   =: L(\cF)(T^0).
}
\end{equation}
This proof that $L(\cF)$ is a Segal sheaf whenever $\cF$ is a Segal sheaf is complete once we show~(\ref{e40}) is a pullback among spaces.

Restricting the diagram~(\ref{e41}) in $\PShv(\bTheta_n)$ along the standard diagonal functor $\bDelta^{\times n} \to \bTheta_n$ determines a pullback diagram
\begin{equation}\label{e42}
\xymatrix{
\cG\bigl( [\bullet],\dots,[\bullet]\bigr)   \ar[rr]  \ar[d]
&&
\cG_+\bigl( [\bullet],\dots,[\bullet]\bigr)  \ar[d]
\\
\cG_-\bigl( [\bullet],\dots,[\bullet]\bigr)  \ar[rr]
&&
\cG_0\bigl( [\bullet],\dots,[\bullet]\bigr)   
}
\end{equation}
in $\PShv(\bDelta^{\times n})$.
Because the diagram~(\ref{e41}) belongs to $\nGpd$, this restricted diagram~(\ref{e42}) belongs to the $\infty$-category $\Gpd^n[\cS]$ of groupoid objects in the $\infty$-category $\Gpd^{n-1}[\cS]$, where $\Gpd^1[\cS] = 1\Gpd[\cS]$ is the $\infty$-category of groupoid objects in $\Spaces$.
Now, let $P:=\bigl([p_1],\dots,[p_{n-1}]\bigr)\in \bDelta^{n-1}$.
Consider the functor $\bDelta \xra{ ( \lag P\rag , \id  ) } \bDelta^{\times n-1}\times \bDelta =\bDelta^{\times n}$.
Consider the restriction 
\begin{equation}\label{e43}
\xymatrix{
\cG^{P} (\bullet)   \ar[rr]  \ar[d]
&&
\cG_+^{P} (\bullet) \ar[d]
\\
\cG_-^{P} (\bullet)   \ar[rr]
&&
\cG_0^{P} (\bullet)   
}
\end{equation}
of the diagram~(\ref{e42}) along this functor.
This diagram~(\ref{e43}) is a pullback diagram among simplicial spaces.  
Because $\cF\in \Shv(\bTheta_n)$ is a Segal sheaf, inspecting the standard diagonal functor $\bDelta^{\times n} \to \bTheta_n$ gives that the vertical arrows in the pullback diagram~(\ref{e43})  are Cartesian fibrations as well as coCartesian fibrations among Segal sheaves on $\bDelta$.
Quillen's Theorem B (of~\cite{fibrations}) applied to~(\ref{e43}) grants that the resulting diagram among spaces
\begin{equation}\label{e44}
\xymatrix{
\bigl|  \cG^{P} (\bullet)  \bigr|  \ar[rr]  \ar[d]
&&
\bigl| \cG_+^{P} (\bullet) \bigr|  \ar[d]
\\
\bigl| \cG_-^{P} (\bullet)   \bigr|  \ar[rr]
&&
\bigl|  \cG_0^{P} (\bullet)   \bigr|
}
\end{equation}
is a pullback.  
Proceeding by induction on $n$ gives that the diagram among spaces
\begin{equation}\label{e45}
\xymatrix{
\bigl|  \cG\bigl( [\bullet],\dots,[\bullet]\bigr)   \bigr|  \ar[rr]  \ar[d]
&&
\bigl|  \cG_+\bigl( [\bullet],\dots,[\bullet]\bigr)   \bigr| \ar[d]
\\
\bigl|  \cG_-\bigl( [\bullet],\dots,[\bullet]\bigr)   \bigr| \ar[rr]
&&
\bigl|  \cG_0\bigl( [\bullet],\dots,[\bullet]\bigr)     \bigr|
}
\end{equation}
is a pullback diagram.
Now, there is a canonical morphism among square diagrams in $\Spaces$ from~(\ref{e45}) to diagram~(\ref{e40}).
By direct examination using Quillen's Theorem A, the diagonal functor $\bDelta^{\times n}\to \bTheta_n$ is final.
%{\color{red} EXPLAIN MORE or CITE THS??}
It follows that this comparison morphism from~(\ref{e45}) to~(\ref{e40}) is an equivalence between square diagrams.
Finally, because~(\ref{e45}) is a pullback diagram, then so too is~(\ref{e40}).

\item 
{\bf Proof of~(2):}
Let $\cF\in \Shv(\bTheta_n)$ be a Segal sheaf.  
By Definition~\ref{def.segal.cov} of the univalence condition, we must show, for each $0\leq i< n$, that the canonical map between spaces
\begin{equation}\label{e27}
L(\cF)(c_i)
\longrightarrow 
L(\cF)\bigl(c_{i}\wr \sE(c_1) \bigr)
\end{equation}
is an equivalence.  
Consider the diagram in $\Shv(\bTheta_n)$:
\[
\xymatrix{
\partial c_i \times \sE(c_{i+1})  \ar[dr]  \ar[rrr]  \ar[ddd]
&
&
&
c_i \times \sE(c_{i+1})  \ar[ddd]  \ar[dl]
\\
&
\partial c_i  \ar[r]  \ar[d]
&
c_i  \ar[d]
&
\\
&
\partial c_i  \ar[r]
&
c_i  
&
\\
\partial c_i  \ar[ur]  \ar[rrr]
&
&
&
c_i \wr \sE(c_1)  .   \ar[ul]  
}
\]
In this diagram, the outer square is a pushout, and the inner square is trivially a pushout.
Furthermore, the diagonal arrows in this square are carried to equivalences by the localization $\Shv(\bTheta_n) \to \Shv^{\sf unv}(\bTheta_n)$. 
Applying the Segal sheaf $L(\cF)$ to this diagram results in the diagram among spaces,
\[
\xymatrix{
L(\cF)( c_i  )  \ar[dr]  \ar[rrr]  \ar[ddd]
&
&
&
L(\cF)( c_i )   \ar[ddd]  \ar[dl]
\\
&
L(\cF)\bigl(c_i \wr \sE(c_1)\bigr)    \ar[r]  \ar[d]
&
L(\cF)\bigl(c_i \times \sE(c_{i+1}) \bigr)    \ar[d]
&
\\
&
L(\cF)\bigl( \partial c_i \times \sE(c_{i+1}) \bigr)  \ar[r]
&
L(\cF)( \partial c_i  )
&
\\
L(\cF)( \partial c_i )    \ar[ur]  \ar[rrr]
&
&
&
L(\cF)( \partial c_i )   , \ar[ul]  
}
\]
in which both the inner and the outer squares are pullback squares.  
Therefore, the down-rightward arrow, which is~(\ref{e27}), is an equivalence provided each of the down-leftward arrow and the up-rightward arrow is an equivalence.  
This is implied by showing, for each $S,T\in \bTheta_n$, that the canonical map between spaces,
\[
L(\cF)(T) := 
\bigl| \cF\bigl(T\times \sE(\bullet) \bigr) \bigr|
\longrightarrow
\bigl| \cF\bigl(T\times \sE(S)\times \sE(\bullet) \bigr) \bigr|
=: L(\cF)\bigl(T\times \sE(S)\bigr)   ~,
\]
is an equivalence. 
So fix $S,T\in \bTheta_n$.

As established in the {\bf Proof of~(1)} above, the presheaf $\cG_T(\bullet):= \cF\bigl(T\times \sE(\bullet)\bigr)$ on $\bTheta_n$ belongs to the full $\infty$-subcategory $\nGpd\underset{\rm Obs~\ref{t19}}\subset \PShv(\bTheta_n)$.  
Inspecting the Definition~\ref{d1} of the functor $\sE\colon \bTheta_n \to \nGpd$, its left Kan extension to $\Cat_n^{\sf strict}$ preserves products.  
Therefore $\sE(S\times \bullet) = \sE(S)\times \sE(\bullet)$.
We conclude an identification of the exponential presheaf, 
\[
\cG_T(\bullet)^{S} := \cG_T(S\times \bullet) = \cF\bigl(T\times\sE(S\times \bullet)\bigr) = \cF\bigl(T\times \sE(S)\times \sE(\bullet)\bigr)   ~, 
\]
and that this is also an $n$-groupoid object in $\Spaces$.  
So it remains to establish the following general claim:
\begin{itemize}
\item[] {\bf ($\dagger$):}
Let $\bTheta_n^{\op}\xra{\cG}\Spaces$ be an $n$-groupoid object.
Let $S\in \bTheta_n$ be an object.
Consider the morphism $\cG \to \cG^S$ between $n$-groupoid objects induced by the unique morphism $S\to \ast$ in $\bTheta_n$.  
The resulting map between colimits
\begin{equation}\label{e29}
| \cG |
\longrightarrow
| \cG^S |
\end{equation}
is an equivalence between spaces.

\end{itemize}
Consider the morphism $s\colon \ast \to S$ that selects the unique object for which there are no non-identity morphisms to it in the $n$-category $S$.  
This morphism is a section of the unique morphism $S\to \ast$ in $\bTheta_n$.  
This determines a retraction $|\cG^S| \xra{s^\ast} |\cG|$, for which~(\ref{e29}) is a section.  
The claim~($\dagger$) then follows upon showing the composite map among spaces
\[
| \cG^S | 
\xra{~s^\ast~}
|\cG|
\xra{ ~(\ref{e29})~ }
|\cG^S |
\]
is equivalent to the identity map.
We do this representably.
More precisely, we construct a filler among strict $n$-categories
\begin{equation}\label{e30}
\xymatrix{
S   \ar[r]^-{\cong} \ar[d]_-{!}  
&
\ast \times S  \ar[r]^-{0\times \id} 
&
c_n \times S  \ar@{-->}[d]_-{H_{S}}
&
\ast \times S  \ar[l]_-{1\times \id}
&
S \ar[l]_-{\cong}   \ar[dll]^-{\id}
\\
\ast  \ar[rr]^-{s}
&&
S   
&&
.
}
\end{equation}
Such a filler determines a filler in the diagram among presheaves on $\bTheta_n$:
\[
\xymatrix{
\cG^S   \ar[rr]^-{0}
&
&
c_n \times \cG^S  \ar@{-->}[d]  
& 
&
\cG^S \ar[ll]_-{1}  
\\
\cG     \ar[u]^-{!^\ast}  
&&
\cG^S     \ar[ll]^-{s^\ast}   \ar[urr]_-{\id}
&&
.
}
\]
The colimit functor $|-|\colon \PShv(\bTheta_n) \to \Spaces$ preserves products, and carries $c_n$ to $\ast$.
Applying this colimit functor to the above diagram results in the sought equivalence between the composite map $|\cG^S| \to |\cG| \to |\cG^S|$ and the identity map.

So we are left to construct a filler, $H_S\colon c_n \times S\to S$, in the diagram~(\ref{e30}).
We do this for all $S\in \bTheta_n$, by induction on $n$.
The base case that $n=0$ is tautological.  
So assume $n>0$.
Write $S=[p](S_1,\dots,S_p)$.
If $p=0$, then $S =\ast$ is terminal and there is a unique $H_S$.
So assume $p>0$.  
Because $S$ is a colimit of its maximal Segal cover, it is enough to define $H_S$ on each term of this colimit, compatibly.
So the functor $H_S$ is determined by declaring, for each consecutive morphism $c_n \to S$ in $\bTheta_n$, what is the composite functor among strict $n$-categories:
\[
H_S\circ (\id,f) \colon c_n
\xra{~ (\id , f) ~}
c_n \times S
\xra{~H_S~}
S  ~.
\]
So let $c_n \xra{f} S$ be a consecutive morphism in $\bTheta_n$ that is injective on sets of objects.  
So $f$ evaluates on objects as $f(0) = i-1$ and $f(1)=i$ for some $0<i\leq p$, and $f$ evaluates on strict $(n-1)$-categories of morphisms as ${\sf Hom}_f(0,1)\colon {\sf Hom}_{c_n}(0,1) = c_{n-1} \xra{f_i} S_i = {\sf Hom}_S( i-1 , i)$ for some consecutive morphism $f_i \colon c_{n-1}\to S_i$ in $\bTheta_{n-1}$.  
We declare the composite functor among strict $n$-categories
\[
H_S\circ (\id,f) \colon c_n
\xra{~ (\id , f) ~}
c_n \times S
\xra{~H_S~}
S
\]
as follows.
On objects, this composite functor is 
\[
H_S\circ (\id , f) (0) = 0
\qquad \text{ and } \qquad
H_S\circ (\id,f)(1) = f(1)   .
\]
On $(n-1)$-categories of morphisms, this composite functor,
\begin{eqnarray}
\nonumber
{\sf Hom}_{c_n}(0,1) 
&
=
&
c_{n-1}
\\
\nonumber
&
\xra{~(\id , f_i)~}
&
c_{n-1} \times S_i 
\\
\nonumber
&
=
&
{\sf Hom}_{c_n}(0,1)\times {\sf Hom}_S\bigl(f(0),f(1)\bigr)
\\
\nonumber
&
\xra{~{\sf Hom}_{H_S}(0,1)~}
&
{\sf Hom}_S\bigl( 0 , f(1)\bigr) = \underset{0<j\leq i} \prod S_j ~,
\end{eqnarray}
is defined by declaring the composition
\[
c_{n-1}
\xra{~(\id,f_i)~}
c_{n-1}\times S_i
\xra{~{\sf Hom}_{H_S}(0,1)~}
\underset{0<j\leq i} \prod S_j
\xra{~\pr~}
S_j
\]
to be the following.
If $j<i$, this latter composition is constant at the object $0\in S_j$.
If $j=i$, this latter composition is $H_{S_i}\circ (\id,f_i)$; note that, because $S_i\in \bTheta_{n-1}$, this has been defined by induction on $n$.
This completes the construction of $H_S$, for each $S\in \bTheta_n$, and therefore completes the proof of~(2).

\item 
{\bf Proof of~(3):}
Let $\cF\in \Shv^{\sf unv}(\bTheta_n)$ be a univalent Segal sheaf.  
We must show the morphism between univalent Segal sheaves
\begin{equation}\label{e28}
(\ref{e24})\colon 
\cF
\longrightarrow
\bigl| \cF\bigl(- \times \sE(\bullet) \bigr)\bigl|
=:L(\cF)
\end{equation}
is an equivalence.  
Let $S,T\in \bTheta_n$.
Consider the projection morphism $T\times \sE(S) \to T$ in the $\infty$-category $\Shv(\bTheta_n)$.
In the case that $T=\ast$, this is the case nearly by definition of the univalence condition.
The case of general $T$ follows from this $T=\ast$ case because the Bousfield localization $\Shv(\bTheta_n) \to \Shv^{\sf unv}(\bTheta_n)$ preserves products (as established in~\cite{rezk-n}).    
Using that $\cF$ is univalent, the resulting map between spaces
\[
\cF(T) 
\longrightarrow
\cF\bigl(T\times \sE(S)\bigr) 
\]
is an equivalence.  
It follows that the morphism~(\ref{e28}) is an equivalence, as desired.

\end{itemize}

\section{Using the formula to prove the main result}
We use the formula of Proposition~\ref{completion} for univalent-completion to prove Theorem~\ref{main.result.1}.
Specifically, we use that formula to prove Lemma~\ref{t9}, Lemma~\ref{t11}, and Lemma~\ref{t12}, then draw Theorem~\ref{main.result.1} as a consequence.

\subsection{Proof of Lemma~\ref{t9}}\label{sec.in.fCat}
Recall from Proposition~\ref{completion} the explicit description of the functor $(-_{|\bTheta_\bullet^{\op}})^{\widehat{~}}_{\sf unv} \colon \Shv(\bTheta_n) \to \Gamma\bigl(\Shv^{\sf unv}(\bTheta_\bullet)\bigr)$.  
Through that description, the problem is to show that, for each Segal sheaf $\cF\in \Shv(\bTheta_n)$, and for each $0\leq i \leq j \leq n$, the canonical functor between $(\infty,j)$-categories
\[
(\cF_{|\bTheta_i^{\op}})^{\widehat{~}}_{\sf unv}
\longrightarrow
(\cF_{|\bTheta_j^{\op}})^{\widehat{~}}_{\sf unv}
\]
is $i$-connective.  
So fix $\cF\in \Shv(\bTheta_n)$ and $0\leq i\leq j \leq n$.  
By Definition~\ref{def.connective} of \emph{$i$-connective}, we must show that, for each $0\leq k \leq i$, the solid diagram among $(\infty,j)$-categories
\begin{equation}\label{e16}
\xymatrix{
\partial c_k  \ar[rr]  \ar[d]  
&&
(\cF_{|\bTheta_i^{\op}})^{\widehat{~}}_{\sf unv} \ar[d]
\\
c_k  \ar[rr]  \ar@{-->}[urr]
&&
(\cF_{|\bTheta_j^{\op}})^{\widehat{~}}_{\sf unv}
}
\end{equation}
can be filled.
So fix such a $0\leq k\leq i$.

Consider the canonical diagram among spaces of functors:
\begin{equation}\label{e17}
\xymatrix{
\Cat_j\bigl( c_k , (\cF_{|\bTheta_i^{\op}})^{\widehat{~}}_{\sf unv} \bigr)  \ar[rr]  \ar[d]
&&
\Cat_j\bigl( c_k , (\cF_{|\bTheta_j^{\op}})^{\widehat{~}}_{\sf unv} \bigr)    \ar[d]
\\
\Cat_j\bigl( \partial c_k , (\cF_{|\bTheta_i^{\op}})^{\widehat{~}}_{\sf unv} \bigr)    \ar[rr]
&&
\Cat_j\bigl( \partial c_k , (\cF_{|\bTheta_j^{\op}})^{\widehat{~}}_{\sf unv} \bigr)     .
}
\end{equation}
Unwinding definitions, the problem of finding a filler in~(\ref{e16}) is identical to showing the resulting map to the pullback
\[
\Cat_j\bigl( c_k , (\cF_{|\bTheta_i^{\op}})^{\widehat{~}}_{\sf unv} \bigr)
\longrightarrow
\Cat_j\bigl( \partial c_k , (\cF_{|\bTheta_i^{\op}})^{\widehat{~}}_{\sf unv} \bigr)
\underset{  \Cat_j\bigl( \partial c_k , (\cF_{|\bTheta_j^{\op}})^{\widehat{~}}_{\sf unv} \bigr)   }  \times
\Cat_j\bigl( c_k , (\cF_{|\bTheta_j^{\op}})^{\widehat{~}}_{\sf unv} \bigr) 
\]
is surjective (on path components).

Recall Definition~\ref{d1} of the functor $\sE\colon \bTheta_n \to \PShv(\bTheta_n)$.  
For each $\ell\leq n$, denote its restriction
\[
\sE_{|\ell} \colon \bTheta_\ell \hookrightarrow \bTheta_n \xra{~\sE~} \PShv(\bTheta_n)~.
\]
The canonical morphism $\partial c_k \to c_k$ then determines, for each $\ell\leq n$, the morphism 
\[
\cF(c_k \times \sE_{|\ell}(\bullet)) 
\longrightarrow
\cF(\partial c_k \times \sE_{|\ell}(\bullet) )
\]
between $\ell$-groupoid objects in $\Spaces$.
Taking colimits results in a commutative diagram among spaces:
\[
\xymatrix{
\bigl| \cF(c_k \times \sE_{|i}(\bullet))  \bigr|  \ar[rr]  \ar[d]
&&
\bigl| \cF(c_k \times \sE_{|j}(\bullet))  \bigr|    \ar[d]
\\
\bigl| \cF(\partial c_k \times \sE_{|i}(\bullet))  \bigr|    \ar[rr]
&&
\bigl| \cF(\partial c_k \times \sE_{|j}(\bullet))  \bigr|  .
}
\]
Through Proposition~\ref{completion}, which identifies univalent-completion in terms of $\sE$, this commutative diagram among spaces is identified as the commutative diagram~(\ref{e17}).
So we must show that the canonical map between spaces
\begin{equation}\label{e18}
\bigl| \cF(c_k \times \sE_{|i}(\bullet))  \bigr| 
\longrightarrow
\bigl| \cF(\partial c_k \times \sE_{|i}(\bullet))  \bigr|  
\underset{  \bigl| \cF(\partial c_k \times \sE_{|j}(\bullet))  \bigr|  }  \times
\bigl| \cF(c_k \times \sE_{|j}(\bullet))  \bigr| 
\end{equation}
is surjective (on path components).
This map~(\ref{e18}) factors as the string of maps,
\begin{eqnarray}
\nonumber
\bigl| \cF(c_k \times \sE_{|i}(\bullet))  \bigr| 
&
\underset{\rm (a)} {~\simeq~}
&
\bigl| \iota^\ast  \cF(c_k \times \sE_{|j}(\bullet))  \bigr| 
\\
\nonumber
&
\underset{\rm (b)} {\xra{~\simeq~}}
&
\Bigl| \iota^\ast \Bigl( \cF(\partial c_k \times \sE_{|j}(\bullet)) 
\underset{   \cF(\partial c_k \times \sE_{|j}(\bullet))  }  \times
\cF(c_k \times \sE_{|j}(\bullet))  \Bigr)  \Bigr| 
\\
\nonumber
&
\underset{\rm (c)} {\xla{~\simeq~}}
&
\Bigl| \iota^\ast \Bigl(  \iota_!\iota^\ast \cF(\partial c_k \times \sE_{|j}(\bullet)) 
\underset{   \cF(\partial c_k \times \sE_{|j}(\bullet))  }  \times
\cF(c_k \times \sE_{|j}(\bullet))  \Bigr)  \Bigr| 
\\
\nonumber
&
\underset{\rm (d)} {\overset{\rm surjective}\longrightarrow}
&
\Bigl| \iota_!\iota^\ast \cF(\partial c_k \times \sE_{|j}(\bullet)) 
\underset{   \cF(\partial c_k \times \sE_{|j}(\bullet))  }  \times
\cF(c_k \times \sE_{|j}(\bullet))  \Bigr| 
\\
\nonumber
&
\underset{\rm (e)} {\xra{~\simeq~}}
&
\bigl|  \iota_! \iota^\ast \cF(\partial c_k \times \sE_{|j}(\bullet))  \bigr|  
\underset{  \bigl| \cF(\partial c_k \times \sE_{|j}(\bullet))  \bigr|  }  \times
\bigl| \cF(c_k \times \sE_{|j}(\bullet))  \bigr| 
\\
\nonumber
&
\underset{\rm (f)} {~\simeq~}
&
\bigl| \cF( \partial c_k \times \sE_{|i}(\bullet))  \bigr|  
\underset{  \bigl| \cF(\partial c_k \times \sE_{|j}(\bullet))  \bigr|  }  \times
\bigl| \cF(c_k \times \sE_{|j}(\bullet))  \bigr|    ~,
\end{eqnarray}
which we now explain.
We use the notation $\iota \colon \bTheta_i \hookrightarrow \bTheta_j$ for the standard fully faithful right adjoint functor.
This functor determines the adjunction $\iota_!\colon \PShv(\bTheta_i) \rightleftarrows \PShv(\bTheta_j)\colon \iota^\ast$, whose left adjoint is given by left Kan extension along $\iota^{\op}$, and whose right adjoint is given by restriction along $\iota^{\op}$.
With this notation, the equivalence~(a) is definitional, from the notation for $\sE_{|\ell}$.
The equivalence~(b) is a trivial pullback.
Because $\iota$ is fully faithful, the unit of this $(\iota_!,\iota^\ast)$-adjunction is an equivalence.
Together with the fact that the right adjoint $\iota^\ast$ preserves pullbacks, this establishes the equivalence~(c).
The map~(d) is the canonical one, induced by restriction along the functor $\iota$.
Because this functor $\iota$ is a right adjoint, for any presheaf $\cX$ on $\bTheta_j$, the canonical map between colimits $| \iota^\ast \cX | \to |\cX|$ is surjective (on path components).  
In particular, the map~(d) is surjective (on path components).
The equivalence~(f) is the fact that left Kan extensions compose, together with the definitional identification $\cF(\partial c_k\times \sE_{|i}(\bullet) ) = \iota^\ast \cF(\partial c_k\times \sE_{|j}(\bullet))$ as that supporting the equivalence~(a).
The equivalence~(e) follows through the same logic, which reduces from presheaves on $\bTheta_n$ to simplicial spaces, as in the final part of {\bf Proof of (1)}, in~\S\ref{sec.correct}.
This completes the proof of Lemma~\ref{t9}.
%{\color{red} BETTER EXPOSITION??}

\subsection{Proof of Lemma~\ref{t11}}\label{sec.unit.proof}
Let $\cF\in \Shv(\bTheta_n)$ be a Segal sheaf.  
Following~\S\ref{sec.unit}, we must show, for each $T\in \bTheta_n$, that the diagram~(\ref{e21}) is a limit diagram.
Through Proposition~\ref{completion}, we can describe each instance of univalent-completion appearing in~(\ref{e21}) by the expression involving $\sE(\bullet)$:
\begin{equation}\label{e31}
\xymatrix{
\cF(T)  \ar[rrrr]  \ar[d]
&&&&
\bigl| \cF\bigl(T\times \sE(\bullet) \bigr) \bigr|   \ar[d]
\\
\cF(T_{<n})  \ar[rr]
&&
\sfN\Bigl( \bigl| \cF\bigl(-\times \sE(\bullet<n) \bigr) \bigr|  \Bigr) (T_{<n})   \ar[rr]
&&
\bigl| \cF\bigl(T_{<n} \times \sE(\bullet) \bigr) \bigr|    .
}
\end{equation}
is a limit diagram.
We do this by induction on $n$.
For $n=0$, this assertion is vacuously true.
So assume $n>0$.  
By induction on $n$, the bottom left horizontal morphism in~(\ref{e31}) is an equivalence.  
Using that the $\infty$-category $\Spaces$ is an $\infty$-topos, Theorem~6.1.0.6 of~\cite{HTT} grants that the square~(\ref{e31}) is a pullback square provided, 
for each $S\in \bTheta_n$, the square
\[
\xymatrix{
\cF(T)  \ar[rr]  \ar[d]
&&
\cF\bigl(T\times \sE(S) \bigr) \ar[d]
\\
\cF(T_{<n})  \ar[rr]
&&
\cF\bigl(T_{<n} \times \sE(S) \bigr)    
}
\]
is a pullback.
To show this square is pullback it is sufficient to show that the square among strict $n$-categories
\begin{equation}\label{e32}
\xymatrix{
T_{<n} \times \sE(S) \ar[rr]  \ar[d]
&&
T\times \sE(S)  \ar[d]
\\
T_{<n} \ar[rr]
&&
T    
}
\end{equation}
is a pushout in $\Shv(\bTheta_n)$.
Because $S\to \sE(S)$ is an epimorphism in $\Shv(\bTheta_n)$, it follows through the Cartesian Bousfield localization of univalent-completion (established in~\cite{rezk-n}) that the top horizontal morphism in~(\ref{e32}) is an epimorphism in $\Shv(\bTheta_n)$.
Thus, the square~(\ref{e32}) among strict $n$-categories is a pushout in $\Shv(\bTheta_n)$ provided the square among strict $n$-categories
\[
\xymatrix{
T_{<n} \times S \ar[rr]  \ar[d]
&&
T\times S  \ar[d]
\\
T_{<n} \ar[rr]
&&
T    
}
\]
is a pushout in $\Shv(\bTheta_n)$.
This latter square is trivially a pushout, after the Cartesian Bousfield localization of Segal completion (established in~\cite{rezk-n}).

\subsection{Proof of Lemma~\ref{t12}}\label{sec.counit.proof}
From the description of the counit in~\S\ref{sec.counit}, we must show, for each $0\leq i\leq n$, that the functor between $(\infty,i)$-categories
\[
{\rm counit}\colon 
\bigl( \sfN(\un{\cC})_{|\bTheta_i^{\op}} \bigr)^{\widehat{~}}_{\sf unv}
\longrightarrow
\cC_i
\]
is an equivalence.  
So fix $0\leq i\leq n$.  
Through the nearly definitional Observation~\ref{t23}, it is enough to show, for each $T\in \bTheta_i$, that the map between spaces
\begin{equation}\label{e35}
\bigl( \sfN(\un{\cC})_{|\bTheta_i^{\op}} \bigr)^{\widehat{~}}_{\sf unv}(T)
\longrightarrow
\cC_i(T)
\end{equation}
is an equivalence.  
So let $T\in \bTheta_i$.
Through Proposition~\ref{completion}, this map~(\ref{e35}) is identified as the map
\begin{equation}\label{e36}
\Bigl|  \sfN(\un{\cC})\bigl(T\times \sE(\bullet_{\leq i})\bigr) \Bigr|
\longrightarrow
\Bigl|  \cC_i\bigl(T\times \sE(\bullet_{\leq i})\bigr) \Bigr|
\xla{~\simeq~}
\cC_i(T)
\end{equation}
in which the leftward arrow is an equivalence because $\cC_i$ is, by definition, univalent.

We establish that the map~(\ref{e36}) is an equivalence using, for each $S\in \bTheta_i$, the following diagram among spaces
\[
\xymatrix{
\sfN(\un{\cC})\bigl(T\times \sE(S)\bigr)    \ar[d]   \ar[drr]  \ar[drrrr]
&&
&&
\\
\sfN(\un{\cC})\bigl(T_{<i}\times \sE(S)\bigr)    \ar[d]   \ar[drr]  \ar[drrrr]
&&
\Bigl|  \sfN(\un{\cC})\bigl(T\times \sE(\bullet_{\leq i})\bigr) \Bigr|    \ar[rr]_-{(a)}  \ar[d]
&&
\cC_i(T)    \ar[d]
\\
\sfN(\un{\cC})\bigl(T_{<i}\times \sE(S_{<i})\bigr)    \ar[drr]  \ar[drrrr]
&&
\Bigl|  \sfN(\un{\cC})\bigl(T_{<i}\times \sE(\bullet_{\leq i})\bigr) \Bigr|  \ar[rr]_-{(b)}  \ar[d]
&&
\cC_i(T_{<i})  \ar[d]
\\
&&
\Bigl|  \sfN(\un{\cC})\bigl(T_{<i}\times \sE(\bullet_{\leq i})\bigr) \Bigr|^{S_{<i}}  \ar[rr]  
&&
\cC_i(T_{<i})^{S_{<i}}  ,
}
\]
which we now explain.
This diagram is a representation of the $\infty$-category $[2]\times [2]$ in the $\infty$-category $\Spaces$, where the second factor selects the vertical sequences.  
\begin{itemize}
\item
For each $u\in [2]$, the representation $\{u\}\times \{0<1\}\to [2]\times [2]\to \Spaces$ selects the map that is induced by the canonical functor $T_{<i}\to T$ between $i$-categories. 

\item 
For each $u\in [2]$, the representation $\{u\}\times \{1<2\}\to [2]\times [2]\to \Spaces$ selects the map that is induced by the canonical functor $S_{<i}\to S$ between $i$-categories. 

\item
For each $v\in [2]$, the representation $\{0<1\}\to \{v\}\times [2]\times [2]\to \Spaces$ selects the map that is induced by the canonical map to a colimit.    

\item
For each $v\in [2]$, the representation $\{1<2\}\times \{v\}\to [2]\times [2]\to \Spaces$ selects the map that is that determined by~(\ref{e36}).  
In particular, the representation (a)$\colon \{1<2\}\times \{0\}\to [2]\times [2]\to \Spaces$ is precisely~(\ref{e36}).

\item
This entire diagram commutes because $\cC_i$ is univalent.  

\end{itemize}
So we must show that the map~(a) is an equivalence.  

We isolate some observations concerning surjectivity. 
\begin{itemize}
\item[{\bf (Surj 1):}]
Observation~\ref{t24} grants that, for $v = 0,1\in [2]$, the representation $\{0<2\}\times\{v\}\to [2]\times[2] \to \Spaces$ selects a map between spaces that is surjective (on path components).

\item[{\bf (Surj 2):}]
Because this functor $\iota\colon \bTheta_{0} \hookrightarrow \bTheta_i$ is a right adjoint, for any presheaf $\cF$ on $\bTheta_i$, the canonical map between colimits $\cF[0] = | \iota^\ast \cF | \to |\cF|$ is surjective (on path components). 
Applying this observation to $\cF=\sfN(\un{\cC})\bigl(T\times \sE(\bullet)\bigr)$ and to $\cF=\sfN(\un{\cC})\bigl(T_{<i}\times \sE(\bullet)\bigr)$ gives that, for $v=0,1\in [2]$, the representation $\{0<1\}\times\{v\}\to [2]\times[2] \to \Spaces$ selects a map between spaces that is surjective (on path components) in the case that $S=[0]$.
The general case for $S\in \bTheta_i$ follows from the 2-out-of-3 property for surjections, again using that $[0]\in \bTheta_i$ is a final object.  

\item[{\bf (Surj 3):}]
After the previous point {\bf (Surj 2)}, commutativity of the square selected by the representation $\{0<1\}\times \{1<2\}\to [2]\times [2]\to \Spaces$ implies the map selected by the representation $\{0<1\}\times\{2\}\to [2]\times [2]\to \Spaces$ is surjective over the image of the diagonal map, which is selected by the representation $\{1\}\times\{1<2\}\to [2]\times [2]\to \Spaces$.

\item[{\bf (Surj 4):}]
By the 2-out-of-3 property for surjections, we conclude that, for $v=0,1\in [2]$, the representation $\{1<2\}\times\{v\}  \to [2]\times[2] \to \Spaces$ selects a map that is surjective (on path components).

\end{itemize}

Now, it follows from Corollary~\ref{t1} that the representation $\{0<2\}\times \{1<2\} \to [2]\times [2] \to \Spaces$ selects a pullback diagram.
As established in {\bf Proof of~(1)} of~\S\ref{sec.correct}, the functor 
\[
\bTheta_i^{\op}\xra{~\sfN(\cC)\bigl(T_{<i}\times \sE(\bullet_{\leq i})\bigr) ~} \Spaces
\]
is an $i$-groupoid object.  
Precisely because $i$-groupoid objects in $\Spaces$ are effective, 
the representation $\{0<1\}\times \{1<2\} \to [2]\times [2]\to \Spaces$ selects a pullback diagram.  
%{\color{red} EXPLAIN MORE??}
Taking the case that $S = [1]$, the two squares $\{0<u\}\times\{0<1\}\to[2]\times [2]\to \Spaces$, for $u=1,2$, being pullback implies, for each point $x \in \sfN(\un{\cC})\bigl(T_{<i}\times \sE([1])\bigr)$, that the induced map between based loop spaces
\[
\Omega_{|x|} \Bigl|  \sfN(\un{\cC})\bigl(T_{<i}\times \sE(\bullet_{\leq i})\bigr) \Bigr| 
\xra{~\Omega (b)~}
\Omega_{(b)|x|} \cC_i(T_{<i})
\]
is an equivalence.  
Using that the $\infty$-topos $\Spaces$ is hypercomplete, we conclude from this, together with the above {\bf (Surj)} observations concerning the bottom horizontal arrow, that the map~(b) is an equivalence between spaces.

The {\bf (Surj)} observations give that each of the representations $\{0<1\}\times \{1\}\to [2]\times [2] \to \Spaces$ and $\{0<2\}\times \{1\}\to [2]\times [2]\to \Spaces$ is surjective (on path components).  
It follows, again, from Corollary~\ref{t1} that the representation $\{0<2\}\times \{0<1\} \to [2]\times [2] \to \Spaces$ selects a pullback diagram.
Established above is that the map~(b) is an equivalence.  
Using, again, that the $\infty$-topos $\Spaces$ is hypercomplete, we conclude from~(b) being an equivalence, that the map~(a) is an equivalence, as desired, provided the representation $\{0<1\}\times\{0<1\} \to [2]\times [2] \to \Spaces$ selects a pullback square.  
So it remains to show just that.

Let $S\to S'$ be a morphism in $\bTheta_i$.
Consider the evident square among strict $i$-categories:
\begin{equation}\label{e37}
\xymatrix{
T_{<i}\times \sE(S')  \ar[rr]  \ar[d]
&&
T \times \sE(S')  \ar[d]
\\
T_{<i}\times \sE(S)    \ar[rr]
&&
T \times \sE(S)  .
}
\end{equation}
Applying the left adjoint functor $(-_{|\bTheta_\bullet^{\op}})^{\widehat{~}}_{\sf unv}\colon \Shv(\bTheta_i) \to \fCat_i$ of~(\ref{e015}) to this square~(\ref{e37}) results in square among flagged $(\infty,i)$-categories.
Using the fact that the Bousfield localization implementing Segal completion is Cartesian (as established in~\cite{rezk-n}), this square in flagged $(\infty,i)$-categories is a pushout.  
From the definition of $\sfN$ as a restricted Yoneda functor, it follows that the square among spaces
\begin{equation}\label{e38}
\xymatrix{
\sfN( \un{\cC} ) \bigl( T \times \sE(S) \bigr)   \ar[rr]  \ar[d]
&&
\sfN(\un{\cC}) \bigl( T \times \sE(S') \bigr)  \ar[d]
\\
\sfN(\un{\cC}) \bigl(  T_{<i}\times \sE(S) \bigr)    \ar[rr]
&&
\sfN(\un{\cC})\bigl( T_{<i}\times \sE(S') \bigr)    
}
\end{equation}
is a pullback.
Using that the $\infty$-category $\Spaces$ is an $\infty$-topos, Theorem~6.1.0.6 of~\cite{HTT} gives that the resulting square involving colimits
\[
\xymatrix{
\sfN( \un{\cC} ) \bigl( T \times \sE(S) \bigr)   \ar[rr]  \ar[d]
&&
\Bigl|  \sfN(\un{\cC}) \bigl( T \times \sE(\bullet) \bigr) \Bigr|     \ar[d]
\\
\sfN(\un{\cC}) \bigl(  T_{<i}\times \sE(S) \bigr)    \ar[rr]
&&
\Bigl|   \sfN(\un{\cC})\bigl( T_{<i}\times \sE(\bullet) \bigr) \Bigr|   .
}
\]
is a pullback.
As this square is the representation $\{0<1\}\times \{0<1\}\to [2]\times [2]\to \Spaces$, this completes this proof.

\subsection{Proof of the main result (Theorem~\ref{main.result.1})}\label{sec.main.proof}
To prove Theorem~\ref{main.result.1}, it is enough to show that the unit and the counit of the adjunction~(\ref{e20}) are equivalences.  
Through the discussion in~\S\ref{sec.unit}, that the unit of the adjunction~(\ref{e14}) is an equivalence is exactly the statement of Lemma~\ref{t11}.
Through the discussion in~\S\ref{sec.counit}, that the counit of the adjunction~(\ref{e14}) is an equivalence is exactly the statement of Lemma~\ref{t12}.

\subsection{Proof of Corollary~\ref{main.corollary}}\label{sec.groupoid}
We prove Corollary~\ref{main.corollary}.  
Recall Definition~\ref{d6} of $n$-groupoid objects, and Definition~\ref{d4} of $n$-flagged $\infty$-groupoids.  

Let $0\leq i \leq n$.
Using that the standard fully faithful functor $\ast = \bTheta_0^{\op} \hookrightarrow \bTheta_i^{\op}$ is a fully faithful left adjoint, restriction and left Kan extension define a localization
\[
\Spaces = \PShv(\bTheta_0)
~\rightleftarrows~
\PShv(\bTheta_i)
\]
in which the left adjoint is fully faithful.
Evidently, this left adjoint factors as the commutative diagram among $\infty$-categories and fully faithful functors there among:
\begin{equation}\label{e26}
\xymatrix{
\Spaces   \ar@{-->}[rr]    \ar@{-->}[drr]  \ar@{-->}[d]
&&
i{\sf Gpd}[\cS]    \ar[d]^-{\rm Obs~\ref{t19}}
\\
\Shv^{\sf unv}(\bTheta_i)   \ar[rr]
&&
\Shv(\bTheta_i)   .
}
\end{equation}
Inspecting the definitions of these full $\infty$-subcategories of $\PShv(\bTheta_i)$ (as Bousfield localizations thereof) reveals that this is a limit diagram.
In other words, if a Segal sheaf on $\bTheta_i$ is both an $n$-groupoid object in $\Spaces$ and univalent complete, then it is a constant presheaf.

\begin{lemma}\label{t21}
For each $0\leq i\leq n$, the adjunction~(\ref{e20}) restricts through Observation~\ref{t19} as an adjunction:
\[
\xymatrix{
i{\sf Gpd}[\cS]     \ar@{-->}@(-,u)[rr]^-{|-|}  \ar[d]_-{\rm Obs~\ref{t19}}
&&
\Spaces \ar[ll]^-{(\ref{e26})}     \ar[d]^-{\rm constant}
\\
\Shv(\bTheta_i)    \ar@(-,u)[rr]^-{(-)^{\widehat{~}}_{\sf unv}} 
&&
\Shv^{\sf unv}(\bTheta_i)  \ar[ll]^-{\rm inclusion}   
}
\]

\end{lemma}

\begin{proof}
The top horizontal arrow $\Spaces \hookrightarrow i{\sf Gpd}[\cS]$ is a right adjoint; its left adjoint is the functor $|-|\colon i{\sf Gpd}[\cS]\to \Spaces$ given by taking the colimit of an $i$-groupoid object $\bTheta_i^{\op}\to \Spaces$.  
Because the straight diagram in the statement of the lemma commutes, there is a canonical lax-commutative diagram:
\[
\xymatrix{
i{\sf Gpd}[\cS]     \ar@(u,u)[rr]^-{|-|}  \ar[d]_-{\rm Obs~\ref{t19}}
&
\Uparrow
&
\Spaces   \ar[d]^-{\rm constant}
\\
\Shv(\bTheta_i)    \ar[rr]^-{(-)^{\widehat{~}}_{\sf unv}} 
&&
\Shv^{\sf unv}(\bTheta_i)      .
}
\]
It remains, then, to show that this lax-commutative diagram is, in fact, a commutative diagram.
So let $\cG\in i{\sf Gpd}[\cS]$ be an $i$-groupoid object.  
We must show that the canonical morphism between univalent Segal sheaves on $\bTheta_i$ under $\cG$,
\[
\cG^{\widehat{~}}_{\sf unv}
\longrightarrow
|\cG|  ~,
\]
is an equivalence.
Through the universal property of the univalent-completion of $\cG$, this is to show that, for each univalent Segal sheaf $\cC$ on $\bTheta_i$, each horizontal morphism in $\Shv(\bTheta_i)$,
\[
\xymatrix{
\cG  \ar[rr]  \ar[d]   \ar@{-->}[drr]
&&
\cC
\\
|\cG|  \ar@{-->}[urr]  \ar@{-->}[rr]
&&
\cC^{\sim}    ,  \ar[u]
}
\]
admits a unique factorization as in the upward diagonal arrow.  
The fully faithful functor $i{\sf Gpd}[\cS] \hookrightarrow \Shv(\bTheta_i)$ is a left adjoint; we will denote its right adjoint as $(-)^{\sim} \colon \Shv(\bTheta_i) \to i{\sf Gpd}[\cS]$.  
The definition of \emph{univalent} Segal sheaves on $\bTheta_i$ is just so that this adjunction restricts as an adjunction
\[
\Spaces 
~\rightleftarrows~
\Shv^{\sf unv}(\bTheta_i) \colon (-)^{\sim}  .
\]
This adjunction determines a unique downward diagonal factorization, and also ensures that the bottom right presheaf $\cC^\sim$ on $\bTheta_i$ is, in fact, constant.  
Because $\cC^\sim$ is constant, there is a unique factorization as in the bottom horizontal arrow.  
This establishes a unique factorization as in the upward diagonal arrow, as desired.
\end{proof}

\begin{remark}
Lemma~\ref{t21} articulates that the univalent-completion of an $i$-groupoid object $\bTheta_i^{\op}\xra{\cG}\Spaces$ is nothing other than its colimit $|\cG|$, regarded as a constant presheaf on $\bTheta_i$ which is thusly univalent and Segal.  

\end{remark}

\begin{lemma}\label{t20}
The adjunction~(\ref{e015}) restricts through Observation~\ref{t19} as an adjunction:
\[
\bigl| (-)_{|\bTheta_\bullet^{\op}} \bigr|  \colon  \nGpd[\cS]  
~\rightleftarrows~
\fGpd_n  \colon   \sfN     ~.
\]

\end{lemma}

\begin{proof}
Lemma~\ref{t21} implies, and identifies how, the restriction of $(-_{|\bTheta_\bullet^{\op}})^{\widehat{~}}_{\sf unv}$ to $\nGpd[\cS]$ factors through $\fGpd_n$.  
It remains to show that the restriction of $\sfN$ to $\fGpd_n$ factors through $\nGpd[\cS]$.

By direct inspection, there is a canonical lax-commutative diagram among $\infty$-categories,
\[
\xymatrix{
\bTheta_n  \ar[rr]  \ar[d]_-{\sE}
&&
\Cat_n  \ar[rr]  
&&
\fCat_n   \ar[rr]
&&
\Gamma\bigl( \Shv^{\sf unv}(\bTheta_\bullet) \to [n]\bigr) 
\\
\nGpd[\cS]  \ar@(d,d)[rrrr]_-{}
&&
\Uparrow
&&
\fGpd_n  \ar[u]  \ar[rr]
&&
\Fun([n],\Spaces)    , \ar[u]
}
\]
where the bottom curved arrow has been established in the first line of this proof, and the right horizontal arrows are the evident fully faithful functors.
For each $0\leq i\leq n$, the standard fully faithful functor $\Spaces \hookrightarrow \Cat_i$ is a right adjoint.  Their left adjoints assemble as a left adjoint to the standard fully faithful functor $\Fun([n],\Spaces) \hookrightarrow \Gamma\bigl( \Shv^{\sf unv}(\bTheta_\bullet) \to [n]\bigr)$ whose image consists of those sections that take values in $\infty$-groupoids.
Replacing the right upward arrow in the above diagram by this left adjoint determines another lax-commutative diagram among $\infty$-categories:
\[
\xymatrix{
\bTheta_n  \ar[rr]  \ar[d]_-{\sE}
&&
\Cat_n  \ar[rr]  
&&
\fCat_n   \ar[rr]
&&
\Gamma\bigl( \Shv^{\sf unv}(\bTheta_\bullet) \to [n]\bigr)  \ar[d]
\\
\nGpd[\cS]  \ar@(d,d)[rrrr]_{}
&&
\Uparrow
&&
\fGpd_n   \ar[rr]
&&
\Fun([n],\Spaces)  .
}
\]
Observation~\ref{t22} implies that this lax-commutative diagram is, in fact, a commutative diagram.  
Now, the functor $\sfN$ is defined as the restricted Yoneda functor along the composite functor $\bTheta_n \to \fCat_n$ appearing in the above diagram.
Commutativity of the above diagram thusly gives that, for each $n$-flagged $\infty$-groupoid $\un{X} = (X_0 \to \dots \to X_n)$, and each $T\in \bTheta_n$, that the canonical monomorphism between spaces
\[
\sfN(\un{X})(\sE(T)) 
\longrightarrow
\sfN(\un{X})(T) 
\]
is an equivalence.  
We conclude that flagged $(\infty,n)$-category $\sfN(\un{X})$ is, in fact, an $n$-flagged $\infty$-groupoid, as desired.  
\end{proof}

\begin{remark}
In Lemma~\ref{t20}, the left adjoint carries an $n$-groupoid object $\bTheta_n^{\op}\xra{\cG} \Spaces$ to the sequence of spaces $|\cG_{|\bTheta_0^{\op}}| \to |\cG_{|\bTheta_1^{\op}}| \to \dots \to |\cG_{|\bTheta_n^{\op}}|$. 
Lemma~\ref{t9} ensures that this sequence of spaces indeed satisfies the requisite connectivity conditions to be an $n$-flagged $\infty$-groupoid.  
The right adjoint in Lemma~\ref{t20} carries an $n$-flagged $\infty$-groupoid $X_0 \to \dots \to X_n$ to its $n$-Cech nerve, elaborated in Remark~\ref{r4}.

\end{remark}

\begin{remark}
Note that the right adjoint in Lemma~\ref{t20} is defined on the full $\infty$-subcategory $\Fun([n],\Spaces)\subset \Gamma\bigl(  \Shv^{\sf unv}(\bTheta_\bullet) \to [n]\bigr)$.
As so, for $X_\bullet \in \Fun([n],\Spaces)$ a sequence of spaces, the counit
\[
\bigl| \sfN(X_\bullet)_{|\bTheta_\bullet^{\op}}  \bigr|  
\longrightarrow
X_\bullet
\]
is the terminal $n$-flagged $\infty$-groupoid over $X_\bullet$.
In this way, this is counit is a relative connective cover of $X_\bullet$.  
Indeed, for $n=1$, and should $X_0 \simeq \ast$, then this morphism is $(\ast \to X_1') \to (\ast \to X_1)$ where $X_1' \to X_1$ is the canonical map from the connected component of $X_1$ through which $\ast \to X_1$ factors.
And for general $n$, yet $X_0\xra{\simeq} \dots \xra{\simeq} X_{n-1} \simeq \ast$, this morphism is $(\ast \to \dots \to  \ast \to X_n') \to (\ast \to \dots \to \ast \to X_n)$ where $X_n'\to X_n$ is the $n$-connective cover of the based space $\ast \to X_n$.

\end{remark}

\begin{proof}[Proof of Corollary~\ref{main.corollary}]
It remains to show that the unit and the counit of the adjunction of Lemma~\ref{t20} are equivalences.  
Through Lemma~\ref{t20}, it is enough to observe that the unit and the counit of the extended adjunction~(\ref{e015}) are equivalences.
That this is so is precisely Lemma~\ref{t11} and Lemma~\ref{t12}, respectively.  
\end{proof}

\end{document}